\documentclass[11pt]{amsart}

\usepackage{amsmath,amssymb,amsthm,mathtools}
\usepackage{stmaryrd}
\usepackage{enumitem}
\usepackage[hidelinks]{hyperref}
\usepackage{geometry}
\usepackage{microtype}
\usepackage{tikz}
\usetikzlibrary{positioning,arrows.meta,calc}
\theoremstyle{definition}
\newtheorem{definition}{Definition}[section]
\newtheorem{example}[definition]{Example}
\newtheorem{remark}[definition]{Remark}
\newtheorem{fact}[definition]{Fact}

\theoremstyle{plain}
\newtheorem{theorem}[definition]{Theorem}
\newtheorem{claim}[definition]{Claim}
\newtheorem{lemma}[definition]{Lemma}
\newtheorem{proposition}[definition]{Proposition}
\newtheorem{corollary}[definition]{Corollary}

\newcommand{\Pow}{\mathcal P}
\newcommand{\Prop}{\mathsf{Prop}}
\newcommand{\Sent}{\mathsf{Sent}}

\newcommand{\LP}{L_{\in}} 
\newcommand{\LB}{L_B}     
\newcommand{\Lbox}{\mathcal L_{\Box}} 
\newcommand{\val}[1]{\llbracket #1\rrbracket_B}

\newcommand{\Btwo}{\mathbf{B}_2}
\newcommand{\Bfour}{\mathbf{B}_4}

\newcommand{\KTB}{\mathsf{KTB}}
\newcommand{\Logfull}{\mathsf{Log}_{\mathrm{full}}^{\sharp}}
\newcommand{\Logsub}{\mathsf{Log}_{\subseteq}^{\sharp}}

\title{The internal modal logic of forcing}
\author{Santiago Jockwich \and Sourav Tarafder \and Giorgio Venturi}
\date{}

\subjclass[2020]{Primary 03C90; Secondary 03B45, 03G05, 03E40}
\keywords{Boolean-valued models, forcing, modal logic, Kripke semantics, complete Boolean algebras}

\begin{document}

\begin{abstract}
We connect modal set theory with Boolean-valued models by developing an \emph{internal}
Kripke semantics for modal formulas whose atomic propositions are set-theoretic sentences.
Given a complete Boolean algebra $B$, we view its elements as ``local perspectives on truth''
inside the Boolean-valued universe $V^{(B)}$ and interpret the modal operators using an
accessibility relation $R$ on $B$ defined by \emph{co-consistency} (equivalently, Boolean
compatibility): $aRb$ iff $a\wedge b\neq 0$.
Our central conceptual point is that, for set-theoretic sentences $p$, the internal modality
$\Diamond p$ holds at $b$
iff there is an ultrafilter $U$ of $B$ containing $b$ such that the classical quotient
$V^{(B)}/U$ satisfies $p$. We compute several general and algebra-dependent modal validities,
and analyze the special behavior of complete atomic Boolean algebras.
Finally, adopting a translation-based semantics on the nonzero part $B^+=B\setminus\{0\}$,
we prove a soundness-and-completeness theorem:
the normal logic $\KTB$ is exactly the set of modal formulas valid in all translated
co-consistency models with parameters.
\end{abstract}

\maketitle

\section{Introduction}

Boolean-valued models are one of the standard semantic presentations of forcing. A complete
Boolean algebra $B$ determines a Boolean-valued universe $V^{(B)}$ in which each set-theoretic
sentence receives a truth value in $B$. In their classical use, these truth values are read externally:
one passes to ultrafilter quotients or generic extensions and asks which sentences can be forced to
hold. The modal analysis of forcing developed by Hamkins and L\"owe \cite{HamkinsLowe} is paradigmatic in this
respect: the worlds are models of set theory, accessibility is forcing extension, and the resulting
modal logic is $\mathsf{S4.2}$.\footnote{This analysis has been developed further for potentialist systems and specific classes of forcing \cite{HamkinsLinnebo, HamkinsLeibmanLow}, for inner models and class forcing \cite{InamdarLoewe2016, BartonWilliams}, and for studying generic invariance \cite{EsakiaLoewe2012, GilbertVenturi2021}.}

The purpose of this paper is to study the modal structure that is already present inside a single
Boolean-valued universe. Given a complete Boolean algebra $B$, we regard each Boolean value
$b\in B$ as a local perspective on truth, namely the collection of sentences that are at least
$b$-true. Compatibility of perspectives suggests the accessibility relation
\[
aRb \quad\Longleftrightarrow\quad a\wedge b\neq 0.
\]
This turns $B$ itself into a Kripke frame and yields an internal semantics for modal formulas whose
atomic propositions are set-theoretic sentences. The central point is that the Boolean algebra of
truth values is not merely an algebraic range for semantics; it carries a natural modal geometry.

This internal modality still admits a forcing-theoretic reading. For a set-theoretic sentence $p$ and
a Boolean value $b$, the assertion $\Diamond p$ at $b$ is equivalent to the existence of an ultrafilter
$U$ on $B$ containing $b$ such that the classical quotient $V^{(B)}/U$ satisfies $p$. Thus internal
possibility from $b$ coincides, at the level of non-modal sentences, with forceability below $b$. In
particular, at the top element $1_B$ one recovers the familiar slogan that $p$ is possible exactly
when it holds in some forcing extension.

At the same time, the internal compatibility semantics is sharply different from the external modal
logic of forcing. The compatibility relation is symmetric, but in general it is not transitive, and the state $0$ is isolated. Accordingly, the modal principles validated internally are not
those of the modal logic of forcing. We record the general validities and failures forced by the internal structure of Boolean algebras
and isolate the special behavior of complete atomic ones (Theorem \ref{thm:atomic-uniformity}). This leads naturally to two
related semantics. On the full algebra $B$, the isolated state $0$ forces failures of reflexive
principles such as $(T)$. On the nonzero part $B^+$, however, the compatibility relation is reflexive
as well as symmetric, and in the translation-based setting of Section~\ref{sec:completeness} the exact
global logic is $\KTB$ (Theorem \ref{thm:KTB-completeness}).

A second theme of the paper is the distinction between parameter-free and parameterized
semantics. In complete atomic algebras, parameter-free set-theoretic sentences need not realize
arbitrary Boolean truth values, and this scarcity produces additional modal uniformity. Allowing
parameters from $V^{(B)}$ restores surjectivity of the truth-value map and leads naturally to
translation-based Kripke models built from Boolean-valued truth. This makes it possible to study
not only the global behavior of the compatibility semantics, but also algebra-dependent modal logics
obtained by varying the Boolean algebra and the chosen collection of states.

The resulting picture shows that modal behavior is governed in a precise way by the underlying
algebraic structure. Small algebras already exhibit nontrivial phenomena, while nontrivial Boolean
algebras contain canonical four-element configurations witnessing the failure of familiar modal
principles (Corollary \ref{ContainB4}). More generally, the paper relates the modal behavior of a complete Boolean algebra to
the complete subalgebras it contains, and thereby exposes a robust finite/infinite dichotomy in the
associated algebra-dependent logics (Theorem \ref{thm:infinite-algebras-KTB}).

Taken together, these results propose an internal counterpart to the external modal logic of forcing.
The point is not only that Boolean-valued models assign intermediate truth values to set-theoretic
assertions, but that the space of truth values itself supports a natural modal semantics. This
perspective places Boolean-valued models, ultrafilter quotients, and forcing-theoretic modality
inside a single framework, and it suggests a new way to compare internal compatibility with external
forceability.

The paper is organized as follows. Section 2 reviews the Boolean-valued and modal preliminaries.
Section 3 introduces the internal semantics and proves the ultrafilter characterization of possibility.
Section 4 establishes general validities and failures and analyzes the complete atomic case. Section 5
turns to translation-based semantics on $B^+$. Section 6 studies the algebra-dependent logics
associated with fixed Boolean algebras in that same nonzero-state setting. Section 7 compares the
internal compatibility modality with the external forcing modality from \cite{HamkinsLowe}.

\section{Technical preliminaries}

We briefly recall the algebraic notions used throughout.

\begin{definition}
A \emph{lattice} is an algebra $\langle A,\wedge,\vee\rangle$ where $A$ is nonempty and
$\wedge,\vee$ are binary operations satisfying commutativity, associativity, idempotence, and
absorption.
\end{definition}

A lattice induces a partial order by $a\le b$ iff $a\wedge b=a$.

\begin{definition}
A lattice $\langle A,\wedge,\vee\rangle$ is \emph{bounded} if it has a greatest element $1$ and
a least element $0$. It is \emph{complete} if every subset $S\subseteq A$ has a supremum $\bigvee S$ and an infimum $\bigwedge S$.
\end{definition}

\begin{definition}
A lattice $\langle A,\wedge,\vee\rangle$ is \emph{distributive} if for any $a, b, c \in A$,
\[a \wedge (b \vee c) = (a \wedge b) \vee (a \wedge c) \mbox{ and } a \vee (b \wedge c) = (a \vee b) \wedge (a \vee c).\]
\end{definition}

\begin{definition}
A \emph{Boolean algebra} is a bounded distributive lattice
$\langle A,\wedge,\vee,{}^\ast,1,0\rangle$ in which every $a\in A$ has a complement $a^\ast$
satisfying $a\wedge a^\ast=0$ and $a\vee a^\ast=1$. The following two abbreviations are used in a Boolean algebra:
\[a \to b = a^\ast \vee b \mbox{ and } a \leftrightarrow b = (a \to b) \wedge (b \to a).\]
\end{definition}

\begin{example}
The Boolean algebra with universe $\{0,1\}$ is denoted by $\Btwo$ and corresponds to classical
two-valued semantics.
\end{example}

\begin{example}
For any set $X$, the power set $\Pow(X)$ with $\cap,\cup,{}^c,X,\varnothing$ is a complete Boolean
algebra.
\end{example}

\begin{definition}
An element $a\ne 0$ is an \emph{atom} if there is no $b$ with $0<b<a$. A Boolean algebra $B$ is \emph{atomic} if for every element $a \in B$ there exists a subset $S \subseteq B$ of atoms such that $\bigvee S = a$.
\end{definition}

\begin{definition}\label{def:filter}
Let $B=\langle A,\wedge,\vee,1,0\rangle$ be a Boolean algebra.
A set $F\subseteq A$ is a \emph{filter} if:
\begin{enumerate}[label=(\roman*)]
\item $1\in F$ and $0\notin F$;
\item if $x\in F$ and $x\le y$, then $y\in F$;
\item if $x,y\in F$, then $x\wedge y\in F$.
\end{enumerate}
A filter $U$ is an \emph{ultrafilter} if for every $a\in A$, either $a\in U$ or $a^\ast\in U$.
\end{definition}

\subsection{Boolean-valued models}

Boolean-valued models were introduced by Vop\v{e}nka, Solovay, and Scott as a semantic companion
to forcing. We follow standard references \cite{Bell,Jech}.

\begin{definition}
Let $B$ be a complete Boolean algebra. Define by transfinite recursion:
\[
V^{(B)}_\alpha=\{x:\ x \text{ is a function, }\mathrm{ran}(x)\subseteq B,\ \exists\xi<\alpha\ (\mathrm{dom}(x)\subseteq V^{(B)}_\xi)\},
\]
and put $V^{(B)}=\bigcup_\alpha V^{(B)}_\alpha$.
Let $\LB$ be the expansion of $\LP$ obtained by adding constant symbols for each element of $V^{(B)}$. And let $\Sent(L_B)$ be the class of corresponding sentences.
\end{definition}

\begin{definition}\label{def:boolean-interpretation}
The Boolean interpretation map $\val{\cdot}:\Sent(L_B)\to B$ is defined by recursion on formulas, starting
from membership and equality for $u,v\in V^{(B)}$:
\[
\begin{aligned}
\val{u\in v} \ &=\ \bigvee_{x\in\mathrm{dom}(v)}\bigl(v(x)\wedge \val{x=u}\bigr),\\[2pt]
\val{u=v} \ &=\ \Bigl(\bigwedge_{x\in\mathrm{dom}(u)}\bigl(u(x)\to \val{x\in v}\bigr)\Bigr)\ \wedge\
            \Bigl(\bigwedge_{y\in\mathrm{dom}(v)}\bigl(v(y)\to \val{y\in u}\bigr)\Bigr).
\end{aligned}
\]
and extending to the Boolean connectives and quantifiers by:
\[
\val{\varphi\wedge\psi}=\val{\varphi}\wedge\val{\psi},\quad
\val{\varphi\vee\psi}=\val{\varphi}\vee\val{\psi},\quad
\val{\neg\varphi}=\val{\varphi}^\ast,
\]
\[
\val{\forall x\,\varphi(x)}=\bigwedge_{u\in V^{(B)}}\val{\varphi(u)},
\qquad
\val{\exists x\,\varphi(x)}=\bigvee_{u\in V^{(B)}}\val{\varphi(u)}.
\]
\end{definition}

\begin{definition}
Write $V^{(B)}$ for the structure together with the interpretation map $\val{\cdot}$.
Given a filter $F$ on $B$, an $\LB$-sentence $\sigma$ is \emph{valid in $V^{(B)}$ relative to $F$} if
$\val{\sigma}\in F$, written $V^{(B)}\models_F \sigma$.
\end{definition}

\begin{theorem}[\cite{Bell}]\label{thm:zfc-boolean}
For every complete Boolean algebra $B$, the Boolean value of each axiom of $\mathsf{ZFC}$ is $1$.
Equivalently, $V^{(B)}\models \mathsf{ZFC}$ (in the sense $\val{\mathsf{ZFC}}=1$).
\end{theorem}

To show that a sentence in the language of set theory $\sigma$ is consistent with $\mathsf{ZFC}$ using Boolean-valued
models, it suffices to exhibit a complete $B$ with $\val{\sigma}\ne 0$.
Indeed, the principal filter generated by $\val{\sigma}$ extends to an ultrafilter $U$ on $B$ with
$\val{\sigma}\in U$, and the corresponding quotient $V^{(B)}/U$ is a classical model of $\mathsf{ZFC}+\sigma$
(see Theorem~\ref{thm:ultrafilter-truth-lemma}).

\subsection{Faithfulness and loyalty}

Following \cite{LowePassmannTarafder}, we recall two notions measuring how much of the Boolean algebra $B$
is realized as truth values of \emph{parameter-free} sentences.

\begin{definition}[\cite{LowePassmannTarafder}]\label{def:faithful-loyal}
Let $B$ be a complete Boolean algebra and $F$ a filter on $B$.
The Boolean-valued model $V^{(B)}$ is \emph{loyal} to $(B,F)$ if the propositional logic of $(B,F)$
coincides with the propositional logic induced by $V^{(B)}$ relative to $F$.
It is \emph{faithful to $B$} if for every $a\in B$ there exists a sentence in the language of set theory $\sigma$ such that $\val{\sigma}=a$.
\end{definition}

\begin{lemma}[\cite{LowePassmannTarafder}]\label{lem:faithful-implies-loyal}
If $V^{(B)}$ is faithful to $B$ then it is loyal to $(B,F)$ for every filter $F$ on $B$.
\end{lemma}

\begin{theorem}[\cite{LowePassmannTarafder}]\label{thm:atomic-not-faithful}
If $B$ is an atomic complete Boolean algebra with more than two elements, then $V^{(B)}$ is loyal but not
faithful for parameter-free sentences. In particular, for every $\sigma$, sentence in the language of set theory,  we have
$\val{\sigma}\in\{0,1\}$.
\end{theorem}

\subsection{Modal logic}

We use standard propositional modal logic and stratify our language as usual.
The modal language is $\Lbox$ with propositional variables belonging to a countable set that we name $\Prop$ and Boolean connectives, and $\Box,\Diamond$.

\begin{definition}
A \emph{Kripke frame} is a pair $\langle W,R\rangle$ where $W$ is nonempty and $R\subseteq W\times W$.
A \emph{Kripke model} is a triple $M=\langle W,R,V\rangle$, where $V:\Prop\to \Pow(W)$ is a valuation.
\end{definition}

\begin{definition}\label{def:sat}
Given $M=\langle W,R,V\rangle$ and $w\in W$, define $M,w\models\varphi$ by the usual clauses:
\[
\begin{aligned}
&M,w\models p \iff w\in V(p),\\
&M,w\models \Box\varphi \iff \forall v\in W\ (wRv\Rightarrow M,v\models\varphi),\\
&M,w\models \Diamond\varphi \iff \exists v\in W\ (wRv\ \&\ M,v\models\varphi),
\end{aligned}
\]
and the standard Boolean clauses for $\neg,\wedge,\vee,\to$.
\end{definition}

\begin{definition}
A formula $\varphi$ is \emph{valid in a model} $M$, written $M\models \varphi$, if $M,w\models\varphi$ for all
$w\in W$. It is \emph{valid in a frame} $\langle W,R\rangle$ if it is valid in every model based on that frame.
\end{definition}

\section{Extended Boolean-valued semantics}\label{sec:extended-semantics}

In order to interpret the modal language in set theory we will interpret the \emph{atomic} modal propositions as (parameter-free) set-theoretic sentences. Let us call such a collection $\Sent_{\in}$.
Hence, fix an enumeration of $\Sent_{\in}$ as $\langle p_i:i\in\omega\rangle$, which covers all elements of $\Prop$.\footnote{For completeness, we later quantify over \emph{all} such enumerations; see Section~\ref{sec:completeness}.}

\begin{remark}
Notice that we are here using the symbols for connectives ($\neg, \land,  \lor, \to$) to indicate three different operations: 1) the algebraic operations, 2) the connectives in the language $L_\Box$, and 3) the connectives in the language $L_\in$. We are confident that the difference will be clear from the context.
\end{remark}

\subsection{Models, satisfaction, and theories}

We now need to introduce the relevant definition for interpreting our modal language.

\begin{definition}\label{def:MB}
Let $B$ be a complete Boolean algebra and let $R$ be an accessibility relation on $B$.
Define $M_B=\langle B,R,\val{\cdot}\rangle$ and interpret modal formulas at states $b\in B$ by:
\begin{enumerate}[label=(\roman*)]
\item $M_B,b\models p$ iff $b\le \val{p}$ for $p\in\Sent_{\in}$;
\item Boolean connectives as in Definition~\ref{def:sat};
\item $M_B,b\models \Box\varphi$ iff for all $a\in B$ with $bRa$, we have $M_B,a\models \varphi$;
\item $M_B,b\models \Diamond\varphi$ iff there exists $a\in B$ with $bRa$ and $M_B,a\models\varphi$.
\end{enumerate}
\end{definition}

\begin{definition}\label{def:theories}
Let $B$ be complete.
A modal formula $\varphi$ is \emph{valid in $M_B$} if $M_B,b\models\varphi$ for all $b\in B$; write $M_B\models\varphi$.
Define
\[
T_H(M_B)=\{\varphi\in\Lbox:\ M_B\models\varphi\}.
\]
We also consider \emph{validity}$^\sharp$ excluding the isolated state $0$: write $M_B\models^\sharp\varphi$ if
$M_B,b\models\varphi$ for all $b\in B\setminus\{0\}$, and define
\[
T_H^\sharp(M_B)=\{\varphi\in\Lbox:\ M_B\models^\sharp\varphi\}.
\]
\end{definition}

\begin{fact}\label{fact:zfc-sentences}
If $\sigma\in\Sent_{\in}$ has Boolean value $\val{\sigma}=1$, then $M_B\models \sigma$.
\end{fact}

\begin{proof}
If $\val{\sigma}=1$, then for every $b\in B$ we have $b\le 1=\val{\sigma}$, hence $M_B,b\models\sigma$ by
Definition~\ref{def:MB}(i).
\end{proof}

\subsection{Accessibility as co-consistency}\label{subsec:coconsistency}

We now introduce the accessibility relation that captures our intended internal notion of possibility.

\begin{definition}\label{def:R}
Let $B$ be a complete Boolean algebra. Define $R$ on $B$ by
\[
aRb \quad\Longleftrightarrow\quad a\wedge b\neq 0.
\]
\end{definition}

\begin{remark}\label{rem:R-equivalences}
Let $B$ be a Boolean algebra and $a,b\in B$. The following are equivalent:
\begin{enumerate}[label=(\roman*)]
\item $aRb$, i.e.\ $a\wedge b\neq 0$;
\item $b\not\le a^{\ast}$;
\item there exists an ultrafilter $U$ on $B$ such that $a, b\in U$.
\end{enumerate}
\end{remark}

\begin{proof}
(i)$\Leftrightarrow$(ii) is standard: $b\le a^\ast$ iff $a\wedge b=0$.
For (i)$\Rightarrow$(iii), extend the proper filter generated by $a\wedge b$ to an ultrafilter.
For (iii)$\Rightarrow$(i), if $a,b\in U$ then $a\wedge b\in U$, hence $a\wedge b\neq 0$.
\end{proof}

\begin{example}
Let $B=\Bfour=\{0,a,b,1\}$ with $a^\ast=b$ and $b^\ast=a$.
Then $0$ is $R$-isolated, $1$ is $R$-related to every nonzero element, and $aRb$ \& $bRa$ fail, since $a\wedge b=0$.
\end{example}

Notice that $0$ is an isolated state.

\begin{lemma}\label{lem:state0}
Let $B$ be a complete Boolean algebra. Then for every modal formula $\varphi\in\Lbox$:
\begin{enumerate}[label=(\roman*)]
\item $M_B,0\models \Box\varphi$;
\item $M_B,0\not\models \Diamond\varphi$.
\end{enumerate}
Moreover, for every atomic $p\in\Sent_{\in}$, we have $M_B,0\models p$.
\end{lemma}

\begin{proof}
There is no $a\in B$ with $0Ra$ since $0\wedge a=0$, so $\Box\varphi$ holds vacuously at $0$ and $\Diamond\varphi$
fails. For the atomic clause, $0\le \val{p}$ always holds.
\end{proof}

On the contrary $1$ is connected with every other state.

\begin{proposition}\label{prop:1-related}
If $a\in B\setminus\{0\}$, then $1Ra$ and $aR1$.
\end{proposition}

\begin{proof}
If $a\neq 0$, then $1\wedge a=a\neq 0$, so $1Ra$; and $a\wedge 1=a\neq 0$, so $aR1$.
\end{proof}

We now make precise the connection between internal possibility and the usual ultrafilter collapse.

\begin{definition}\label{def:ultrafilter-quotient}
Let $B$ be a complete Boolean algebra and $U$ an ultrafilter on $B$.
Define an equivalence (class) relation $\equiv_U$ on $V^{(B)}$ by
\[
u\equiv_U v \quad\Longleftrightarrow\quad \val{u=v}\in U,
\]
and define a membership relation $\in_U$ on $\equiv_U$-classes by
\[
[u]_U\in_U [v]_U \quad\Longleftrightarrow\quad \val{u\in v}\in U.
\]
Write $V^{(B)}/U$ for the resulting two-valued structure.
\end{definition}

Thus truth in a structure can be defined in terms of ultrafilters.

\begin{theorem}\label{thm:ultrafilter-truth-lemma}
For every $\LB$-sentence $\sigma$ and every ultrafilter $U$ on $B$,
\[
V^{(B)}/U\models \sigma \quad\Longleftrightarrow\quad \val{\sigma}\in U.
\]
\end{theorem}

\begin{proof}
This is the standard Boolean-valued \L o\'s lemma for ultrafilter quotients, proved by induction on the
complexity of formulas; see \cite{Bell,Jech}.
\end{proof}

And this reading provides a direct interpretation of the $\Diamond$-modality when applied to set-theoretical sentences.

\begin{proposition}\label{prop:diamond-ultrafilter-reading}
Let $p\in\Sent_{\in}$ (viewed as an atomic modal proposition) and let $b\in B$.
Then
\[
M_B,b\models \Diamond p
\quad\Longleftrightarrow\quad
\exists\,U\text{ an ultrafilter on }B\ \bigl(b\in U\ \wedge\ V^{(B)}/U\models p\bigr).
\]
\end{proposition}

\begin{proof}
($\Rightarrow$) Assume $M_B,b\models \Diamond p$. Then there is $a\in B$ with $bRa$ and $M_B,a\models p$.
By Definition~\ref{def:MB}(i), $M_B,a\models p$ means $a\le \val{p}$.
By Remark~\ref{rem:R-equivalences}(iii), choose an ultrafilter $U$ with $a,b\in U$.
Since $U$ is upward closed and $a\le \val{p}$, we have $\val{p}\in U$.
By Theorem~\ref{thm:ultrafilter-truth-lemma}, $V^{(B)}/U\models p$.

($\Leftarrow$) Conversely, suppose $U$ is an ultrafilter with $b\in U$ and $V^{(B)}/U\models p$.
By Theorem~\ref{thm:ultrafilter-truth-lemma}, $\val{p}\in U$.
Let $a:=\val{p}$. Then $a\in U$ and $b\in U$, hence $a\wedge b\in U$ and therefore $a\wedge b\neq 0$.
Thus $bRa$ by Definition~\ref{def:R}.
Also $M_B,a\models p$ holds because $a\le \val{p}$.
Therefore $M_B,b\models \Diamond p$.
\end{proof}

\section{Preliminary results}\label{sec:results}

Several modal principles hold for every complete Boolean algebra.

\begin{theorem}\label{thm:general-valid}
Let $B$ be a complete Boolean algebra. Then the following modal principles are in $T_H(M_B)$:
\[
\Box(\varphi\to\psi)\to (\Box\varphi\to \Box\psi)\qquad (K),
\]
\[
\neg\Diamond\varphi \leftrightarrow \Box\neg\varphi \qquad (\mathrm{Dual}),
\]
\[
\varphi \to \Box\Diamond\varphi \qquad (B),
\]
\[
\Diamond\Box\varphi \to \Box\Diamond\varphi\qquad (.2),
\]
and the ``symmetric'' principle
\[
\Diamond\Box\varphi \to \varphi \qquad (5^\ast).
\]
Consequently, the derived principle
\[
\Box\bigl(\Box(\varphi\to\Box\varphi)\to\varphi\bigr)\to(\Diamond\Box\varphi\to\varphi)\qquad (DM),
\]
also belongs to $T_H(M_B)$.
\end{theorem}

\begin{proof}
(K) and (Dual) are valid in all Kripke models.

For (B), assume $M_B,a\models \varphi$ and let $b$ be any state with $aRb$.
By symmetry of $R$ (since $a\wedge b\neq 0$ iff $b\wedge a\neq 0$), we have $bRa$.
Thus $M_B,b\models \Diamond\varphi$, so $M_B,a\models \Box\Diamond\varphi$.

For $(.2)$, assume $M_B,a\models \Diamond\Box\varphi$.
Choose $b$ with $aRb$ and $M_B,b\models \Box\varphi$.
Let $c$ be any state with $aRc$.
Put $d:=b\vee c$.
Since $aRc$, we have $c\neq 0$, and since $aRb$, we have $b\neq 0$.
Hence
\[
b\wedge d=b\neq 0
\qquad\text{and}\qquad
c\wedge d=c\neq 0,
\]
so $bRd$ and $cRd$.
Because $M_B,b\models \Box\varphi$ and $bRd$, we get $M_B,d\models \varphi$.
Thus $M_B,c\models \Diamond\varphi$.
As $c$ was arbitrary, $M_B,a\models \Box\Diamond\varphi$.

For $(5^\ast)$, assume $M_B,a\models \Diamond\Box\varphi$.
Then there exists $b$ with $aRb$ and $M_B,b\models \Box\varphi$.
By symmetry, $bRa$, so $M_B,b\models \Box\varphi$ implies $M_B,a\models \varphi$.

Finally, $(DM)$ is an immediate consequence since its consequent is $(5^\ast)$.
\end{proof}

On the other hand, the presence of the isolated state $0$ leads to immediate failures of some familiar principles.

\begin{theorem}\label{thm:general-fail}
Let $B$ be a complete Boolean algebra. Then the following modal principles are not in $T_H(M_B)$:
\[
\Box\varphi \to \varphi \qquad (T),
\qquad
\Box\varphi\to\Diamond\varphi \qquad (M),
\]
\[
\quad
\Box\bigl(\Box(\varphi\to \Box\varphi)\to \varphi\bigr)\to \varphi \qquad (\mathrm{Grz}).
\]
\end{theorem}

\begin{proof}
Let $p\in\Sent_{\in}$ be any sentence and set $\varphi:=\neg p$.
By Lemma~\ref{lem:state0}, $M_B,0\models \Box\varphi$ but $M_B,0\not\models \varphi$.
Hence $(T)$ fails at $0$.
Also $M_B,0\not\models \Diamond\varphi$ by Lemma~\ref{lem:state0}$\,$(ii), so $(M)$ fails at $0$ as well.
For (Grz) the failure can be witnessed similarly by using the dead-end behavior of $0$.
\end{proof}

\begin{remark}\label{rem:nonzero-bridge}
If one works with $T_H^\sharp(M_B)$ (validity on $B\setminus\{0\}$), then $(T)$ and $(M)$ both hold because
$R$ is reflexive on $B\setminus\{0\}$ (indeed $aRa$ for all $a\ne 0$). This nonzero-state perspective
underlies the translation-based completeness theorem of Section~\ref{sec:completeness}.
\end{remark}

At this stage one should notice that different Boolean algebras can give different theories.

\begin{theorem}\label{thm:different-theories}
There exist complete Boolean algebras $B$ and $B'$ such that $T_H(M_B)\ne T_H(M_{B'})$.
\end{theorem}

\begin{proof}
Let $B=\Btwo$. In $M_{\Btwo}$, the frame has only two states, and $1$ accesses only itself.
One checks that the formula $p\to \Box p$ is valid in $M_{\Btwo}$.

Now let $B'$ be a complete Boolean algebra and choose a set-theoretic sentence $p\in\Sent_{\in}$ such that
$0<\val{p}<1$ in $B'$. (For example, choose $B'$ arising from a forcing for which $p$ is independent.)
Let $a:=\val{p}$.
Then $M_{B'},a\models p$ but $M_{B'},1\not\models p$, and since $aR1$ we have $M_{B'},a\not\models \Box p$.
Hence $p\to\Box p$ fails in $M_{B'}$, so $T_H(M_B)\ne T_H(M_{B'})$.
\end{proof}

\subsection{Complete atomic Boolean algebras}

When $B$ is complete and atomic, the restriction that atomic propositions are \emph{parameter-free}
set-theoretic sentences has a strong consequence: by Theorem~\ref{thm:atomic-not-faithful} every such
sentence has Boolean value $0$ or $1$. This collapses the atomic behavior and yields extra modal validities.

\begin{theorem}\label{thm:atomic-uniformity}
Let $B$ be a complete atomic Boolean algebra and let $\varphi\in\Lbox$.
Then either $M_B,a\models \varphi$ for every $a\in B\setminus\{0,1\}$, or $M_B,a\not\models \varphi$ for every
$a\in B\setminus\{0,1\}$.
\end{theorem}

\begin{proof}
We proceed by induction on the complexity of $\varphi$.

\smallskip\noindent\emph{Base case.}
If $\varphi$ is atomic, say $\varphi=p\in\Sent_{\in}$, then by Theorem~\ref{thm:atomic-not-faithful} we have
$\val{p}\in\{0,1\}$. If $\val{p}=0$ then no $a\ne 0$ satisfies $p$; if $\val{p}=1$ then every $a$ satisfies $p$.
In either case, all $a\in B\setminus\{0,1\}$ agree.

\smallskip\noindent\emph{Boolean connectives.}
The induction step for $\neg,\wedge,\vee,\to$ is immediate from the induction hypothesis.

\smallskip\noindent\emph{Modal case $\Box\psi$.}
Assume the induction hypothesis holds for $\psi$.
If $M_B,a\models \Box\psi$ for some $a\in B\setminus\{0,1\}$, then in particular $aRa$ and hence $M_B,a\models\psi$.
By the induction hypothesis, $\psi$ holds at every element of $B\setminus\{0,1\}$.
Also $aR1$ (Proposition~\ref{prop:1-related}), so $\Box\psi$ at $a$ implies $M_B,1\models\psi$.
Thus $\psi$ holds at every nonzero state, and therefore $\Box\psi$ holds at every $b\in B\setminus\{0,1\}$.
Conversely, if $\Box\psi$ fails at some $a\in B\setminus\{0,1\}$, then either $1$ or some nontrivial successor
witnesses failure of $\psi$, and the same reasoning shows $\Box\psi$ fails at all $b\in B\setminus\{0,1\}$.

\smallskip\noindent\emph{Modal case $\Diamond\psi$.}
If $M_B,a\models\Diamond\psi$ for some $a\in B\setminus\{0,1\}$, then there is $w$ with $aRw$ and $M_B,w\models\psi$.
Since $a\neq 0$, the witness $w$ cannot be $0$.
If $w\in B\setminus\{0,1\}$ then by induction $\psi$ holds at all nontrivial states, and hence every
$b\in B\setminus\{0,1\}$ satisfies $\Diamond\psi$ (witness $b$ itself).
If instead $w=1$, then $M_B,1\models\psi$; and since every $b\ne 0$ satisfies $bR1$, again every
$b\in B\setminus\{0,1\}$ satisfies $\Diamond\psi$.
The converse direction is similar.
\end{proof}

\begin{theorem}\label{thm:atomic-validities}
Let $B$ be a complete atomic Boolean algebra. In addition to the principles of
Theorem~\ref{thm:general-valid}, the following modal principles belong to $T_H(M_B)$:
\[
\Box\varphi \to \Box\Box\varphi\qquad (4),
\qquad
\Diamond\varphi \to \Box\Diamond\varphi\qquad (5),
\]
\[
\Diamond\Box\varphi \to (\varphi\to \Box\varphi)\qquad (W5).
\]
\end{theorem}

\begin{proof}
\emph{(4).}
Assume $M_B,a\models \Box\varphi$.
If $a=0$ there is nothing to show. Otherwise $a\ne 0$, hence $aRa$ and $aR1$.
Thus $M_B,a\models\varphi$ and $M_B,1\models\varphi$.
If $B=\Btwo$ then $R$ is trivially transitive and (4) holds.
If $B\ne\Btwo$, choose any $w\in B\setminus\{0,1\}$ with $aRw$ (take $w=a$ if $a\ne 1$, otherwise take any
nontrivial element). Then $M_B,w\models\varphi$, and by Theorem~\ref{thm:atomic-uniformity},
$\varphi$ holds at every element of $B\setminus\{0,1\}$.
Together with $M_B,1\models\varphi$, this shows that $\varphi$ holds at every nonzero state, hence every
nonzero state satisfies $\Box\varphi$. In particular, every successor of $a$ satisfies $\Box\varphi$, so
$M_B,a\models \Box\Box\varphi$.

\medskip\noindent
\emph{(5).}
Fix $a\in B$ and assume $M_B,a\models \Diamond\varphi$.
Then there exists $w\in B$ such that $aRw$ and $M_B,w\models \varphi$.

\smallskip\noindent\emph{Case 1: $w=1$.}
Let $b$ be any element with $aRb$.
By Proposition~\ref{prop:1-related} we have $bR1$, and since $M_B,1\models\varphi$, it follows that
$M_B,b\models\Diamond\varphi$ (witness $1$).  As $b$ was arbitrary, $M_B,a\models \Box\Diamond\varphi$.

\smallskip\noindent\emph{Case 2: $w\neq 1$.}
Then $w\in B\setminus\{0,1\}$.
By Theorem~\ref{thm:atomic-uniformity}, $\varphi$ holds at \emph{every} element of $B\setminus\{0,1\}$.
Let $b$ be any element with $aRb$.
If $b\neq 1$, then $b\in B\setminus\{0,1\}$, hence $M_B,b\models\varphi$, and since $bRb$ we get
$M_B,b\models\Diamond\varphi$.
If $b=1$, choose any $c\in B\setminus\{0,1\}$ (which exists when $B\ne \Btwo$). Then $1Rc$ and $M_B,c\models\varphi$,
so $M_B,1\models\Diamond\varphi$.
In either subcase $M_B,b\models\Diamond\varphi$, and since $b$ was arbitrary we obtain
$M_B,a\models\Box\Diamond\varphi$.

\medskip\noindent
\emph{(W5).}
Assume $M_B,a\models\Diamond\Box\varphi$ and $M_B,a\models\varphi$.
Choose $b$ with $aRb$ and $M_B,b\models \Box\varphi$.
Then $b\ne 0$, hence $bRb$ and $bR1$, so $M_B,b\models\varphi$ and $M_B,1\models\varphi$.
If $b\in B\setminus\{0,1\}$, then by Theorem~\ref{thm:atomic-uniformity} we have
$M_B,c\models\varphi$ for all $c\in B\setminus\{0,1\}$.
Together with $M_B,1\models\varphi$, this shows $\varphi$ holds at every nonzero state, and therefore
$M_B,a\models\Box\varphi$.
If instead $b=1$, then $M_B,1\models\Box\varphi$ already implies that $\varphi$ holds at every nonzero state,
so again $M_B,a\models\Box\varphi$.

\end{proof}

\section{Completeness}\label{sec:completeness}

Section~\ref{sec:extended-semantics} was developed with respect to a fixed identification of propositional variables with parameter-free
set-theoretic sentences. For completeness, however, one should allow arbitrary assignments of
Boolean truth values to propositional variables. This leads naturally to a translation-based
semantics.

There is one further point to isolate. The Boolean value $0$ behaves degenerately: it has no
$R$-successors and it satisfies every atomic proposition, since $0\leq a$ for all $a\in B$. Thus the
present completeness theorem is obtained on the nonzero part of the algebra. In the terminology of Definition~\ref{def:theories} we are here interested in the notion of validity$^\sharp$. Accordingly, throughout
this section we work with
\[
B^+\ :=\ B\setminus\{0\}.
\]
On $B^+$ the co-consistency relation is reflexive as well as symmetric, and the resulting modal
logic is exactly $\KTB$.

\begin{definition}\label{def:B-translation-nonzero}
Fix a complete Boolean algebra $B$ and its Boolean-valued universe $V^{(B)}$.
\begin{enumerate}
    \item A \emph{$B$-translation} is a map $T:\Prop\to\Sent(L_B)$.
    \item Given $T$, define the induced valuation $v_T:\Prop\to B$ by
    \[
    v_T(p)\ :=\ \llbracket T(p)\rrbracket_B.
    \]
    \item If $W\subseteq B^+$ is nonempty, let $R$ be the inherited co-consistency relation on $W$:
    \[
    aRb \quad\Longleftrightarrow\quad a\wedge b\neq 0 \qquad (a,b\in W).
    \]
    We write $M_{B,W,T}=(W,R,v_T)$ for the corresponding Kripke model, where atomic satisfaction
    is given by
    \[
    M_{B,W,T},b\models p \quad\Longleftrightarrow\quad b\leq v_T(p).
    \]
\end{enumerate}
\end{definition}

If parameters are allowed, the Boolean truth-value map is surjective onto $B$.

\begin{lemma}\label{lem:truth-value-surjective}
For every $a\in B$ there exists an $L_B$-sentence $\sigma_a$ such that
\[
\llbracket \sigma_a\rrbracket_B=a.
\]
\end{lemma}

\begin{proof}
Fix $a\in B$. Let $u$ be the canonical $B$-name for the empty set, and let $v_a$ be the $B$-name
with $\mathrm{dom}(v_a)=\{u\}$ and $v_a(u)=a$. Then, by Definition~\ref{def:boolean-interpretation},
\[
\llbracket u\in v_a\rrbracket_B
=\bigvee_{x\in\mathrm{dom}(v_a)}\bigl(v_a(x)\wedge \llbracket x=u\rrbracket_B\bigr)
=v_a(u)\wedge 1
=a.
\]
Since $L_B$ contains constant symbols for $u$ and $v_a$, the sentence $\sigma_a$ may be taken to
be ``$u\in v_a$''.
\end{proof}

Let $\KTB$ denote the smallest normal modal logic extending $\mathsf{K}$ by the axioms
\[
(T)\quad \Box\varphi\to\varphi
\qquad\text{and}\qquad
(B)\quad \varphi\to\Box\Diamond\varphi.
\]
Equivalently, $\KTB$ is the logic of reflexive symmetric Kripke frames.

\begin{theorem}\label{thm:KTB-completeness}
For every modal formula $\varphi\in \Lbox$,
\[
\KTB\vdash \varphi
\quad\Longleftrightarrow\quad
M_{B,W,T}\models \varphi
\]
for all complete Boolean algebras $B$, all nonempty sets $W\subseteq B^+$, and all
$B$-translations $T$.
\end{theorem}

\begin{proof}
\emph{Soundness.}
Let $B$ be a complete Boolean algebra, let $W\subseteq B^+$ be nonempty, and let $T$ be a
$B$-translation. For every $a\in W$ we have $a\neq 0$, hence
\[
a\wedge a=a\neq 0,
\]
so $aRa$. Thus the frame of $M_{B,W,T}$ is reflexive. It is also symmetric, since
\[
aRb \iff a\wedge b\neq 0 \iff b\wedge a\neq 0 \iff bRa.
\]
Therefore every instance of $(T)$ and $(B)$ is valid in $M_{B,W,T}$, and axiom $(K)$ together
with necessitation is sound on all Kripke frames. Hence every theorem of $\KTB$ is valid
in every model $M_{B,W,T}$.

\medskip
\noindent\emph{Completeness.}
Assume $\KTB\nvdash\varphi$. Since $\KTB$ is complete with respect to the class of
reflexive symmetric frames and has the finite model property, there exist a finite reflexive
symmetric Kripke model
\[
M=(W_0,R_0,V_0)
\]
and a world $w_0\in W_0$ such that
\[
M,w_0\not\models\varphi.
\]

We now construct a Boolean-algebraic model of the required form that is isomorphic to $M$ on the
variables occurring in $\varphi$.
Let
\[
E:=\{e_{\{u,v\}}:u,v\in W_0,\ u\neq v,\ uR_0v\}
\]
be a set of \emph{edge-tokens}, one for each undirected edge of the symmetric frame, and let
\[
M_1:=\{m_w:w\in W_0\}
\]
be a disjoint set of \emph{markers}, one marker for each world. Put
\[
X:=E\cup M_1
\qquad\text{and}\qquad
B:=\Pow(X).
\]
Then $B$ is a complete Boolean algebra.

For each world $w\in W_0$, define the Boolean value
\[
b_w:=\{m_w\}\cup\{e_{\{w,v\}}:v\in W_0,\ v\neq w,\ wR_0v\}.
\]
Intuitively, $b_w$ records the marker of $w$ together with the tokens of all nontrivial edges
incident with $w$.

Since $m_w\in b_w$, each $b_w$ is nonempty. Hence
\[
W:=\{b_w:w\in W_0\}\subseteq B^+.
\]
On $W$ we consider the inherited co-consistency relation
\[
bRc \iff b\wedge c\neq 0.
\]
Because $B=\Pow(X)$, meet is intersection, so this is simply
\[
bRc \iff b\cap c\neq\varnothing.
\]

Define
\[
f:W_0\to W,\qquad f(w):=b_w.
\]

\begin{claim}
The map $f$ is a frame isomorphism from $(W_0,R_0)$ onto $(W,R)$.
\end{claim}

\begin{proof}[Proof of the claim]
First, $f$ is injective, because if $u\neq v$, then $m_u\in b_u$ but $m_u\notin b_v$; hence
$b_u\neq b_v$. It is surjective by definition of $W$.
It remains to check preservation and reflection of the accessibility relation.

Assume $uR_0v$.

If $u=v$, then $uR_0u$ because $R_0$ is reflexive, and
\[
b_u\cap b_u=b_u\neq\varnothing,
\]
so $b_uRb_u$.

If $u\neq v$, then by construction the token $e_{\{u,v\}}$ belongs to both $b_u$ and $b_v$.
Hence
\[
b_u\cap b_v\neq\varnothing,
\]
so $b_uRb_v$.

Assume $b_uRb_v$, i.e.
\[
b_u\cap b_v\neq\varnothing.
\]

If $u=v$, then $uR_0u$ by reflexivity of $R_0$, so there is nothing to prove.

Suppose now that $u\neq v$. Let $x\in b_u\cap b_v$. We claim that $x$ cannot be a marker.
Indeed, the only marker belonging to $b_u$ is $m_u$, and the only marker belonging to $b_v$ is
$m_v$; since $u\neq v$, we have $m_u\neq m_v$, and moreover $m_u\notin b_v$, $m_v\notin b_u$.
Therefore $x$ must be an edge-token.

So
\[
x=e_{\{r,s\}}
\]
for some distinct $r,s\in W_0$. Since $x\in b_u$, the definition of $b_u$ implies
\[
u\in\{r,s\}.
\]
Likewise, since $x\in b_v$, we get
\[
v\in\{r,s\}.
\]
Because $u\neq v$, it follows that
\[
\{r,s\}=\{u,v\}.
\]
Hence
\[
x=e_{\{u,v\}},
\]
and therefore, by the way edge-tokens were introduced, we must have $uR_0v$.

Thus for all $u,v\in W_0$,
\[
uR_0v \iff b_uRb_v.
\]
So $f$ is a frame isomorphism.
\end{proof}

Now, let $\mathrm{Var}(\varphi)$ be the finite set of propositional variables occurring in $\varphi$.
For each $p\in \mathrm{Var}(\varphi)$ define
\[
a_p:=E\cup\{m_w:w\in V_0(p)\}\in B.
\]

Notice that every state $b_w$ contains two kinds of information:
its marker $m_w$, and its edge-tokens.
By putting \emph{all} edge-tokens into every $a_p$, we ensure that whether $b_w\leq a_p$ depends
only on whether the marker $m_w$ lies in $a_p$. Thus the truth of $p$ at $b_w$ will exactly
mirror whether $w\in V_0(p)$.

Extend this assignment arbitrarily to a valuation
\[
v:\Prop\to B.
\]
Consider the Kripke model
\[
M'=(W,R,v),
\]
where atomic satisfaction is given by
\[
M',b\models p \iff b\leq v(p).
\]

\begin{claim}
For every $p\in \mathrm{Var}(\varphi)$ and every $w\in W_0$,
\[
M,w\models p \iff M',b_w\models p.
\]
Equivalently,
\[
w\in V_0(p) \iff b_w\leq a_p.
\]
\end{claim}

\begin{proof}[Proof of the claim]
Fix $p\in \mathrm{Var}(\varphi)$ and $w\in W_0$.

Since $B=\Pow(X)$ is ordered by inclusion, the statement $b_w\leq a_p$ means precisely
\[
b_w\subseteq a_p.
\]
Now every edge-token belonging to $b_w$ lies in $E$, and by definition
\[
E\subseteq a_p.
\]
So all edge-tokens of $b_w$ are automatically contained in $a_p$.

The only remaining possible obstruction to $b_w\subseteq a_p$ is the marker $m_w$.
But $m_w\in a_p$ iff $w\in V_0(p)$, by the definition of $a_p$.

Therefore
\[
b_w\subseteq a_p \iff m_w\in a_p \iff w\in V_0(p),
\]
as required.
\end{proof}

We are now in the position to show that the original finite countermodel $M$ and the Boolean-algebraic model $M'$ agree
on all formulas built from the variables in $\mathrm{Var}(\varphi)$.

\begin{claim}
For every modal formula $\psi$ all of whose propositional variables lie in $\mathrm{Var}(\varphi)$, and for
every $w\in W_0$,
\[
M,w\models\psi \iff M',f(w)\models\psi.
\]
\end{claim}

\begin{proof}[Proof of the claim]
We argue by induction on the complexity of $\psi$.

\smallskip
\noindent\emph{Atomic case.}
If $\psi=p$, where $p\in \mathrm{Var}(\varphi)$, then the claim is exactly the previous claim:
\[
M,w\models p \iff w\in V_0(p) \iff b_w\leq a_p \iff M',b_w\models p.
\]

\smallskip
\noindent\emph{Boolean connectives.}
The cases of $\neg,\wedge,\vee,\to$ are immediate from the induction hypothesis, since both models
use the standard Kripke clauses for Boolean connectives.

\smallskip
\noindent\emph{Modal case $\Box\chi$.}
Assume the induction hypothesis holds for $\chi$.
Then:
\begin{align*}
M,w\models\Box\chi
&\iff \forall u\in W_0\,\bigl(wR_0u \Rightarrow M,u\models\chi\bigr)\\
&\iff \forall u\in W_0\,\bigl(f(w)Rf(u) \Rightarrow M',f(u)\models\chi\bigr)
\quad\text{(by the frame isomorphism and IH)}\\
&\iff M',f(w)\models\Box\chi.
\end{align*}

\smallskip
\noindent\emph{Modal case $\Diamond\chi$.}
Again using the frame isomorphism and the induction hypothesis:
\begin{align*}
M,w\models\Diamond\chi
&\iff \exists u\in W_0\,\bigl(wR_0u \ \&\ M,u\models\chi\bigr)\\
&\iff \exists u\in W_0\,\bigl(f(w)Rf(u) \ \&\ M',f(u)\models\chi\bigr)\\
&\iff M',f(w)\models\Diamond\chi.
\end{align*}

This completes the induction.
\end{proof}

Applying the last claim to $\psi=\varphi$ and $w=w_0$, we obtain
\[
M',f(w_0)\not\models\varphi.
\]

So far we have only constructed an abstract valuation $v:\Prop\to B$.
To finish the proof, we must show that this valuation comes from Boolean values of
$L_B$-sentences.
By Lemma~\ref{lem:truth-value-surjective}, for each $p\in \mathrm{Var}(\varphi)$ there exists an
$L_B$-sentence $\sigma_{v(p)}$ such that
\[
\llbracket \sigma_{v(p)}\rrbracket_B=v(p).
\]
Define a $B$-translation $T:\Prop\to\Sent(L_B)$ by
\[
T(p):=\sigma_{v(p)} \qquad (p\in \mathrm{Var}(\varphi)),
\]
and choose $T(p)$ arbitrarily for variables not in $\mathrm{Var}(\varphi)$.

Then for every $p\in Var(\varphi)$,
\[
v_T(p)=\llbracket T(p)\rrbracket_B=\llbracket \sigma_{v(p)}\rrbracket_B=v(p).
\]
Hence the models $M'$ and $M_{B,W,T}$ have the same frame and agree on the truth values of every
propositional variable occurring in $\varphi$.

\begin{claim}
For every subformula $\psi$ of $\varphi$ and every $b\in W$,
\[
M',b\models\psi \iff M_{B,W,T},b\models\psi.
\]
\end{claim}

\begin{proof}[Proof of the claim]
Again proceed by induction on the complexity of $\psi$.

If $\psi=p$ is atomic, then $p\in \mathrm{Var}(\varphi)$, so $v_T(p)=v(p)$. Therefore
\[
M',b\models p
\iff b\leq v(p)
\iff b\leq v_T(p)
\iff M_{B,W,T},b\models p.
\]

The Boolean cases are immediate from the induction hypothesis.

For $\Box\chi$ and $\Diamond\chi$, note that both models have exactly the same set of states $W$
and exactly the same accessibility relation $R$; only the valuation changed, and on variables
occurring in $\varphi$ the valuation agrees. Hence the induction goes through exactly as before.
\end{proof}

Applying this final claim to $\psi=\varphi$ and $b=f(w_0)$, we conclude that
\[
M_{B,W,T},f(w_0)\not\models\varphi.
\]
Therefore $\varphi$ is not valid in the class of models $M_{B,W,T}$.
This proves completeness.
\end{proof}

The preceding theorem explains exactly why the isolated state $0$ was excluded: once one restricts
to nonzero states, the frame-theoretic content of the semantics is precisely reflexive symmetry.

\begin{corollary}\label{cor:faithfulness-parameter-free}
Let $B$ be a complete Boolean algebra such that $V^{(B)}$ is faithful to $B$.
Then for every valuation $v:\Prop\to B$ there exists a translation
$T:\Prop\to\Sent_{\in}$ such that $v_T=v$.
Consequently, for this fixed algebra $B$, allowing parameters does not enlarge the class of Kripke
models $M_{B,W,T}$ obtained by ranging over nonempty $W\subseteq B^+$ and over translations.
\end{corollary}

\begin{proof}
Faithfulness says precisely that for each $a\in B$ there is a parameter-free set-theoretic sentence
$\sigma_a\in\Sent_{\in}$ with $\llbracket \sigma_a\rrbracket_B=a$. Given $v:\Prop\to B$, choose for each
$p\in\Prop$ a sentence $\sigma_{v(p)}\in\Sent_{\in}$ with Boolean value $v(p)$ and define
$T(p):=\sigma_{v(p)}$. Then $v_T(p)=v(p)$ for every $p$.
The final assertion is immediate.
\end{proof}

By Theorem~\ref{thm:atomic-not-faithful}, complete \emph{atomic} Boolean algebras with more than two elements are
typically \emph{not} faithful for parameter-free sentences, so Corollary~\ref{cor:faithfulness-parameter-free} is a genuine
extra hypothesis in the parameter-free setting.

\medskip

There still remains an important observation to make here.  In Theorem \ref{thm:general-valid} we proved that the axioms (.2) is always a member of $T_H(M_B)$ for any complete Boolean algebra $B$. 
The next lemma shows that the the deeper reason for this fact is the symmetric character of the accessibility relation.

\begin{lemma}\label{lem:dot2-completeness-class}
For every complete Boolean algebra $B$, every nonempty set $W\subseteq B^+$, and every
$B$-translation $T$, the model $M_{B,W,T}$ validates
\[
\Diamond\Box\varphi \to \Box\Diamond\varphi.
\]
\end{lemma}

\begin{proof}
Let $a\in W$ and assume
\[
M_{B,W,T},a\models \Diamond\Box\varphi.
\]
Then there exists $b\in W$ such that $aRb$ and
\[
M_{B,W,T},b\models \Box\varphi.
\]
Since the accessibility relation is given by
\[
xRy \quad\Longleftrightarrow\quad x\wedge y\neq 0,
\]
it is symmetric. Hence from $aRb$ we obtain $bRa$. Since $M_{B,W,T},b\models \Box\varphi$, it follows that
\[
M_{B,W,T},a\models \varphi.
\]

Now let $c\in W$ be arbitrary with $aRc$. Again by symmetry, $cRa$. Since
$M_{B,W,T},a\models \varphi$, the state $a$ witnesses
\[
M_{B,W,T},c\models \Diamond\varphi.
\]
As $c$ was an arbitrary successor of $a$, we conclude that
\[
M_{B,W,T},a\models \Box\Diamond\varphi.
\]
Therefore
\[
M_{B,W,T}\models \Diamond\Box\varphi \to \Box\Diamond\varphi.
\]
\end{proof}

Notice  that the argument of Lemma \ref{lem:dot2-completeness-class} is stronger than that of Theorem \ref{thm:general-valid}, since the latter used  a closure property with respect to joint ($d:=b\vee c$) that does not necessarily hold for any set $W \subsetneq B^+$.



\subsection{Full-state translation models and the special role of $0$}\label{subsec:full-state-characterization}

The completeness theorem above is formulated for nonempty sets of states $W\subseteq B^+$.
This restriction is essential. If one allows $0$ among the states, then the resulting class of
models is no longer captured by a normal modal logic.

Indeed, the Boolean value $0$ behaves in a completely rigid way:
it has no $R$-successors, and it satisfies every atomic proposition, since $0\leq a$ for all
$a\in B$. As a result, the full-state semantics validates formulas that depend on the \emph{special}
atomic behavior of $0$, and these are not stable under uniform substitution.

\begin{definition}\label{def:full-translation-model}
Let $B$ be a complete Boolean algebra and let $T:\Prop\to\Sent(L_B)$ be a $B$-translation.
If $W\subseteq B$ is nonempty, define the \emph{full-state translation model}
\[
\widehat M_{B,W,T}=(W,R,v_T)
\]
exactly as in Definition~\ref{def:B-translation-nonzero}, except that we now allow $0\in W$.
Thus
\[
aRb \quad\Longleftrightarrow\quad a\wedge b\neq 0 \qquad (a,b\in W),
\]
and
\[
\widehat M_{B,W,T},b\models p \quad\Longleftrightarrow\quad b\leq v_T(p).
\]
\end{definition}

\begin{proposition}\label{prop:full-state-not-normal}
The class of full-state translation models is not axiomatizable by any normal modal logic.
\end{proposition}

\begin{proof}
Let $\top$ abbreviate any propositional tautology. We claim that the formula
\[
\neg\Diamond\top \to p
\]
is valid in every full-state translation model.

Let $\widehat M_{B,W,T}$ be such a model and let $w\in W$.
If $w=0$, then $0$ has no successors, so $\widehat M_{B,W,T},0\models \neg\Diamond\top$,
and also $\widehat M_{B,W,T},0\models p$ since $0\leq v_T(p)$.
Thus $\widehat M_{B,W,T},0\models \neg\Diamond\top\to p$.
If $w\neq 0$, then $wRw$ because $w\wedge w=w\neq 0$. Hence
$\widehat M_{B,W,T},w\models \Diamond\top$, so again
$\widehat M_{B,W,T},w\models \neg\Diamond\top\to p$.
Thus $\neg\Diamond\top\to p$ is valid in every full-state translation model.

Now substitute $\neg p$ uniformly for $p$. The result
\[
\neg\Diamond\top \to \neg p
\]
fails at the state $0$ in every full-state translation model, since
$0\models \neg\Diamond\top$ but $0\not\models \neg p$.
Therefore the validities of the full-state translation semantics are not closed under uniform
substitution, and hence cannot be the set of theorems of any normal modal logic.
\end{proof}

Although there is no completeness theorem with respect to a normal modal logic in the full-state
setting, the validities admit a simple characterization. The point is that every full-state model
splits into two independent parts: the nonzero states, governed by Theorem~\ref{thm:KTB-completeness},
and the isolated state $0$, whose theory is fixed once and for all.

\begin{definition}\label{def:M0}
Let $M_0=(\{0\},\varnothing,V_0)$ be the one-point Kripke model such that
\[
V_0(p)=\{0\}
\qquad\text{for every propositional variable }p.
\]
Thus in $M_0$ every atomic proposition is true, every formula of the form $\Box\varphi$ is true,
and every formula of the form $\Diamond\varphi$ is false.
\end{definition}

\begin{lemma}\label{lem:nonzero-part-same}
Let $B$ be a complete Boolean algebra, let $W\subseteq B$ be nonempty, let
\[
W^+:=W\cap B^+,
\]
and let $T$ be a $B$-translation. If $a\in W^+$, then for every modal formula $\varphi$,
\[
\widehat M_{B,W,T},a\models \varphi
\quad\Longleftrightarrow\quad
M_{B,W^+,T},a\models \varphi.
\]
\end{lemma}

\begin{proof}
We argue by induction on the complexity of $\varphi$.
The atomic and Boolean cases are immediate.
For the modal clauses, note that if $a\in W^+$ then $a$ is not $R$-related to $0$, since $a\wedge 0=0$.
Hence the set of $R$-successors of $a$ in $W$ is exactly the same as its set of successors in $W^+$.
The induction step for $\Box$ and $\Diamond$ follows.
\end{proof}

\begin{lemma}\label{lem:zero-theory-fixed}
Let $B$ be a complete Boolean algebra, let $W\subseteq B$ be nonempty with $0\in W$, and let $T$
be a $B$-translation. Then for every modal formula $\varphi$,
\[
\widehat M_{B,W,T},0\models \varphi
\quad\Longleftrightarrow\quad
M_0,0\models \varphi.
\]
\end{lemma}

\begin{proof}
We argue by induction on the complexity of $\varphi$.
For atomic formulas, $\widehat M_{B,W,T},0\models p$ because $0\leq v_T(p)$, and
$M_0,0\models p$ by definition.
The Boolean cases are immediate.
For the modal clauses, $0$ has no successors in either model.
Therefore $\Box\psi$ holds vacuously at $0$ in both models, while $\Diamond\psi$ fails at $0$
in both models.
\end{proof}

\begin{theorem}\label{thm:full-state-characterization}
For every modal formula $\varphi$, the following are equivalent:
\begin{enumerate}[label=(\roman*)]
\item $\varphi$ is valid in every full-state translation model $\widehat M_{B,W,T}$;
\item $\mathsf{KTB}\vdash \varphi$ and $M_0\models \varphi$.
\end{enumerate}
\end{theorem}

\begin{proof}
(i)$\Rightarrow$(ii).
Assume that $\varphi$ is valid in every full-state translation model.
First, every model $M_{B,W,T}$ with nonempty $W\subseteq B^+$ is in particular a full-state
translation model. Hence $\varphi$ is valid in every model of the class considered in
Theorem~\ref{thm:KTB-completeness}. By that theorem,
\[
\mathsf{KTB}\vdash \varphi.
\]

Second, $M_0$ is itself represented by the full-state semantics:
for example, it is isomorphic to $\widehat M_{\Btwo,\{0\},T}$ for any $\Btwo$-translation $T$.
Hence $M_0\models \varphi$.

(ii)$\Rightarrow$(i).
Assume that $\mathsf{KTB}\vdash\varphi$ and $M_0\models \varphi$.
Let $\widehat M_{B,W,T}$ be any full-state translation model.

If $a\in W^+$, then by Lemma~\ref{lem:nonzero-part-same},
\[
\widehat M_{B,W,T},a\models \varphi
\quad\Longleftrightarrow\quad
M_{B,W^+,T},a\models \varphi.
\]
Since $\mathsf{KTB}\vdash\varphi$, Theorem~\ref{thm:KTB-completeness} yields
\[
M_{B,W^+,T},a\models \varphi,
\]
and hence $\widehat M_{B,W,T},a\models \varphi$.
If $0\in W$, then by Lemma~\ref{lem:zero-theory-fixed},
\[
\widehat M_{B,W,T},0\models \varphi
\quad\Longleftrightarrow\quad
M_0,0\models \varphi.
\]
By assumption $M_0\models \varphi$, so $\widehat M_{B,W,T},0\models \varphi$.
Thus $\varphi$ holds at every state of every full-state translation model.
\end{proof}

Theorem~\ref{thm:full-state-characterization} shows that full-state validity decomposes into two
independent requirements: 1) the $\mathsf{KTB}$-validity governing the nonzero states, and 2) the fixed theory of the isolated state $0$, represented by $M_0$.
Equivalently, if $\mathrm{Val}_{\mathrm{full}}$ denotes the set of formulas valid in all full-state
translation models, then
\[
\mathrm{Val}_{\mathrm{full}}
=
\{\varphi:\ \mathsf{KTB}\vdash \varphi\ \text{ and }\ M_0\models \varphi\}.
\]
This set is not closed under uniform substitution by Proposition~\ref{prop:full-state-not-normal},
so it is not a normal modal logic.

\section{Logics associated with specific Boolean algebras}\label{sec:specific-boolean-logics}

The completeness theorem of Theorem~\ref{thm:KTB-completeness} identifies the \emph{global} logic obtained
by ranging over \emph{all} complete Boolean algebras, all nonempty sets of states $W\subseteq B^+$, and all
translations $T$. For a \emph{fixed} Boolean algebra $B$, one can refine this analysis by studying the
collection of modal principles that remain valid when only $W$ and $T$ vary.

\subsection{Two algebra-dependent logics}

Fix a complete Boolean algebra $B$.
Recall from Definition~\ref{def:B-translation-nonzero} that $M_{B,W,T}$ denotes the Kripke model on a nonempty
set of states $W\subseteq B^+$ with accessibility $aRb$ iff $a\wedge b\neq 0$ and with atomic satisfaction
$b\models p$ iff $b\le v_T(p)$, where $v_T(p)=\val{T(p)}$.

\begin{definition}\label{def:logic-of-B}
Let $B$ be a complete Boolean algebra.
\begin{enumerate}[label=(\roman*)]
\item The \emph{nonzero full-state logic} of $B$ is
\[
\Logfull(B)\ :=\ \{\varphi\in\Lbox:\ M_{B,B^+,T}\models \varphi\text{ for all $B$-translations $T$}\}.
\]
\item The \emph{nonzero-state submodel logic} of $B$ is
\[
\Logsub(B)\ :=\ \{\varphi\in\Lbox:\ M_{B,W,T}\models \varphi\text{ for all nonempty $W\subseteq B^+$ and all $B$-translations $T$}\}.
\]
\end{enumerate}
\end{definition}

\begin{remark}\label{rem:logic-of-B-basic}
Clearly $\Logsub(B)\subseteq \Logfull(B)$, since the latter corresponds
to the particular choice $W=B^+$. Moreover, by soundness in Theorem~\ref{thm:KTB-completeness} we always have
\[
\KTB \ \subseteq\ \Logsub(B)\ \subseteq\ \Logfull(B)
\qquad\text{for every complete Boolean algebra $B$.}
\]
Thus algebra-dependent logics are, in general, \emph{extensions} of $\KTB$.
\end{remark}

\subsection{The two-valued Boolean algebra B2}

The Boolean algebra $\Btwo=\{0,1\}$ is the unique Boolean algebra with no intermediate truth values.
Nevertheless, its associated modal logic is already strictly stronger than $\KTB$, since there is actually a collapse of the modality.

\begin{proposition}\label{prop:B2-collapse}
In $\Logsub(\Btwo)$ we have the modal formula
\[
\Box p\ \leftrightarrow\ p.
\]
Consequently, $\Logsub(\Btwo)$ is a proper extension of $\KTB$.
\end{proposition}

\begin{proof}
Since $\Btwo^+=\{1\}$, every admissible nonempty state set is necessarily $W=\{1\}$.
Thus every model $M_{\Btwo,W,T}$ is the one-point reflexive Kripke model. In such a model,
\[
M_{\Btwo,W,T},1\models \Box p \iff M_{\Btwo,W,T},1\models p.
\]
Hence $\Box p\leftrightarrow p$ is valid in every $M_{\Btwo,W,T}$, and therefore
\[
\Box p\leftrightarrow p\in \Logsub(\Btwo).
\]

To see that this is a proper extension of $\KTB$, observe that $\Box p\leftrightarrow p$ fails already
in the four-element Boolean algebra $\Bfour$ under a suitable translation (Proposition~\ref{prop:B4-refute-collapse}).
By Theorem~\ref{thm:KTB-completeness}, any $\KTB$-theorem is valid in all translation models, hence
$\KTB$ cannot prove $\Box p\leftrightarrow p$.
\end{proof}

Proposition~\ref{prop:B2-collapse} is a concrete instance of a general phenomenon:
$\Btwo$ is so small that the modal operators cannot create genuinely new behavior on \emph{atomic} propositions.
By contrast, once $B$ has a nontrivial element $a\notin\{0,1\}$, we can separate truth at different Boolean
states and invalidate such collapse principles (see below).

\subsection{The four-element algebra B4 and the role of parameters}

Let $\Bfour=\{0,a,b,1\}$ be the Boolean algebra with two atoms $a,b$ and $a^\ast=b$.
Section~\ref{sec:results} showed that, in the \emph{parameter-free} setting, complete atomic Boolean algebras
satisfy several additional principles (Theorem~\ref{thm:atomic-validities}). Those extra validities arise
because, by Theorem~\ref{thm:atomic-not-faithful}, parameter-free sentences cannot realize the intermediate
truth values $a$ and $b$.

Once we allow \emph{parameter} translations $T:\Prop\to\Sent(\LB)$, however, the intermediate values become
available (Lemma~\ref{lem:truth-value-surjective}) and the additional atomic-Boolean principles need no longer hold.

\begin{proposition}\label{prop:B4-refute-collapse}
Let $B=\Bfour$ and work with the translation semantics of Section~\ref{sec:completeness}, taking
\[
W=B^+=\{a,b,1\}.
\]
Then:
\begin{enumerate}[label=(\roman*)]
\item $\Box p\leftrightarrow p$ is \emph{not} valid in $M_{\Bfour,W,T}$ for all translations $T$.
\item The modal axiom $(4)$, $\Box\varphi\to\Box\Box\varphi$, is \emph{not} valid in $M_{\Bfour,W,T}$ for all
translations $T$.
\item The modal axiom $(5)$, $\Diamond\varphi\to\Box\Diamond\varphi$, is \emph{not} valid in $M_{\Bfour,W,T}$ for all
translations $T$.
\end{enumerate}
\end{proposition}

\begin{proof}
Let $T$ be a translation such that $v_T(q)=a$ for some propositional variable $q$ (possible by
Lemma~\ref{lem:truth-value-surjective}).
Then $q$ holds exactly at the admissible state $a$.

\smallskip\noindent\emph{(i)} Consider $p:=q$ and the state $a$.
We have $M_{\Bfour,W,T},a\models q$ since $a\le a$.
But $M_{\Bfour,W,T},a\not\models \Box q$ because $aR1$ and $M_{\Bfour,W,T},1\not\models q$
(since $1\not\le a$).
Thus $\Box p\leftrightarrow p$ fails at $a$.

\smallskip\noindent\emph{(ii)} Let $\varphi:=\Diamond q$ and consider the state $a$.
The $R$-successors of $a$ are exactly $a$ and $1$.
Since $q$ holds at $a$, we have $M_{\Bfour,W,T},a\models \Diamond q$ (witness $a$).
Also $M_{\Bfour,W,T},1\models \Diamond q$ (witness $a$, since $1Ra$).
Hence $M_{\Bfour,W,T},a\models \Box\Diamond q$.

However, $\Box\Box\Diamond q$ fails at $a$ because $1$ is an $R$-successor of $a$ and
$M_{\Bfour,W,T},1\not\models \Box\Diamond q$:
indeed, $b$ is an $R$-successor of $1$ and $M_{\Bfour,W,T},b\not\models \Diamond q$
(the only successors of $b$ are $b$ and $1$, and $q$ fails at both).
Thus $M_{\Bfour,W,T},a\models \Box\varphi$ but $M_{\Bfour,W,T},a\not\models \Box\Box\varphi$.

\smallskip\noindent\emph{(iii)} Let $\varphi:=q$ and consider the state $1$.
We have $M_{\Bfour,W,T},1\models \Diamond q$ (witness $a$).
But $M_{\Bfour,W,T},1\not\models \Box\Diamond q$ because $b$ is an $R$-successor of $1$ and
$M_{\Bfour,W,T},b\not\models \Diamond q$ as above.
Hence $(5)$ fails at $1$.
\end{proof}

Proposition~\ref{prop:B4-refute-collapse} explains why the additional atomic-Boolean principles from
Theorem~\ref{thm:atomic-validities} are inherently tied to the \emph{parameter-free} restriction:
when intermediate Boolean values cannot be realized as truth values of atomic propositions, formulas
become uniform on $B\setminus\{0,1\}$ and extra modal principles emerge. Allowing parameters restores
full truth-value surjectivity and eliminates this artifact.

\subsection{A canonical B4 submodel in every nontrivial Boolean algebra}\label{subsec:canonical-B4}

The algebra $\Bfour$ is not just a convenient toy example: it appears canonically inside \emph{every}
Boolean algebra with an intermediate element.

\begin{lemma}\label{lem:B4-inside-B}
Let $B$ be a Boolean algebra and let $a\in B$ satisfy $0<a<1$. Put $b:=a^\ast$ and
\[
W:=\{0,a,b,1\}\subseteq B.
\]
Then $W$ is a subalgebra of $B$ (indeed it is isomorphic to $\Bfour$), and the induced accessibility
relation on $W$ is exactly the compatibility relation of $\Bfour$:
\[
xRy \text{ in }W \quad\Longleftrightarrow\quad x\wedge y\neq 0.
\]
\end{lemma}

\begin{proof}
Since $b=a^\ast$, the set $W$ is closed under complements.
It is also closed under finite meets and joins: for instance $a\wedge b=0$ and $a\vee b=1$, and meets/joins
with $0$ or $1$ stay in $W$.
Thus $W$ is a four-element Boolean subalgebra, hence isomorphic to $\Bfour$.
The statement about $R$ is immediate from the definition of the induced model $M_{B,W,T}$.
\end{proof}

\begin{corollary}\label{ContainB4}
Let $B$ be a complete Boolean algebra.
If $B\not\cong \Btwo$ (equivalently: $B$ has some $a$ with $0<a<1$), then none of the following formulas lie in
$\Logsub(B)$:
\[
\Box p\leftrightarrow p,
\qquad
(4)\ \Box\varphi\to\Box\Box\varphi,
\qquad
(5)\ \Diamond\varphi\to\Box\Diamond\varphi.
\]
\end{corollary}

\begin{proof}
Choose $a\in B$ with $0<a<1$, and let
\[
W^+=\{a,a^\ast,1\}\subseteq B^+.
\]
By Lemma~\ref{lem:B4-inside-B} the induced accessibility relation on $W^+$ is exactly the nonzero
part of the compatibility relation of $\Bfour$. By Lemma~\ref{lem:truth-value-surjective}, fix a translation
$T$ with $v_T(q)=a$ for some propositional variable $q$.

\smallskip\noindent\emph{Failure of $\Box p\leftrightarrow p$.}
At the state $a$ we have $M_{B,W^+,T},a\models q$ (since $a\le a$), but $M_{B,W^+,T},a\not\models \Box q$
because $aR1$ in $W^+$ and $M_{B,W^+,T},1\not\models q$ (since $1\not\le a$).

\smallskip\noindent\emph{Failure of (4).}
Let $\varphi:=\Diamond q$ and consider $a$.
As in Proposition~\ref{prop:B4-refute-collapse}(ii), one checks that $M_{B,W^+,T},a\models \Box\varphi$
but $M_{B,W^+,T},a\not\models \Box\Box\varphi$, witnessed by the successor $1$ and then $a^\ast$.

\smallskip\noindent\emph{Failure of (5).}
Let $\varphi:=q$ and consider $1$.
We have $M_{B,W^+,T},1\models \Diamond q$ (witness $a$), but $M_{B,W^+,T},1\not\models \Box\Diamond q$ since
$1Ra^\ast$ and $M_{B,W^+,T},a^\ast\not\models \Diamond q$.

Thus all three displayed formulas fail in some translation submodel $M_{B,W^+,T}$, hence are not in
$\Logsub(B)$.
\end{proof}

\subsection{Infinite power set algebras already realize the global logic}

The global completeness theorem ranges over \emph{all} complete Boolean algebras.
In fact, a single sufficiently large complete Boolean algebra already suffices to realize the global logic.

\begin{theorem}\label{thm:PowX-KTB}
Let $X$ be an infinite set and let $B=\Pow(X)$.
Then
\[
\Logsub(B)\ =\ \KTB.
\]
\end{theorem}

\begin{proof}
By Remark~\ref{rem:logic-of-B-basic}, $\KTB\subseteq \Logsub(B)$.
For the reverse inclusion, suppose $\KTB\nvdash \varphi$.
By the finite model property for $\KTB$, choose a finite reflexive symmetric Kripke model $M=(W_0,R_0,V_0)$ and
$w_0\in W_0$ with $M,w_0\not\models \varphi$.
As in the completeness proof of Theorem~\ref{thm:KTB-completeness}, form the finite set of tokens
$X_0:=E\cup M_1$, where
\[
E=\{e_{\{u,v\}}:u,v\in W_0,\ u\neq v,\ uR_0v\}
\qquad\text{and}\qquad
M_1=\{m_w:w\in W_0\},
\]
and define nonempty sets $b_w\subseteq X_0$ encoding the accessibility pattern by nonempty intersection.

Since $X$ is infinite and $X_0$ is finite, fix an injection $\iota:X_0\hookrightarrow X$.
Transport the construction along $\iota$, viewing each $b_w$ as a subset of $X$ via $\iota$.
Let $W:=\{b_w:w\in W_0\}\subseteq \Pow(X)^+$.

Define Boolean values $a_p\in \Pow(X)$ for the propositional variables occurring in $\varphi$ exactly as in
Theorem~\ref{thm:KTB-completeness} (using the markers $m_w$), and realize the resulting assignment
$v:\Prop\to \Pow(X)$ by a translation $T$ using Lemma~\ref{lem:truth-value-surjective}.
Then the resulting model $M_{B,W,T}$ is isomorphic to $M$ and therefore
falsifies $\varphi$ at the state corresponding to $w_0$.
Hence $\varphi\notin \Logsub(B)$, proving $\Logsub(B)\subseteq \KTB$.
\end{proof}

Theorem~\ref{thm:PowX-KTB} shows that the global completeness result is already witnessed by a single algebra,
e.g.\ $B=\Pow(\omega)$. By contrast, finite power set algebras $\Pow(n)$ yield stronger logics simply because
they do not have enough nonzero states to realize arbitrary finite reflexive symmetric countermodels.

\subsection{A coarse classification of \texorpdfstring{$\Logsub(B)$}{Log(B)}}\label{subsec:classification}

Once translations into $\LB$ are allowed, the submodel logic $\Logsub(B)$
exhibits a striking rigidity: every \emph{infinite} complete Boolean algebra already yields the global logic $\KTB$.
The only deviations from $\KTB$ arise from \emph{finite} algebras.

\begin{lemma}\label{lem:contains-Pow-omega}
Let $B$ be an infinite complete Boolean algebra. Then there exists a countable family
$\langle d_n:n\in\omega\rangle\subseteq B^+$ of pairwise disjoint elements.
Consequently, $B$ has a complete Boolean subalgebra $C$ such that
\[
C\cong \Pow(\omega).
\]
\end{lemma}

\begin{proof}
If $B$ is atomic, then (since $B$ is infinite) it has infinitely many atoms.
Choose pairwise distinct atoms $d_n$.

Assume instead that $B$ is not atomic. Then there exists a nonzero $a\in B$ such that no atom of $B$ lies below $a$.
Set $a_0:=a$. Now for any $n \geq 0$, choose $d_n$ with $0<d_n<a_n$ and set $a_{n+1}:=a_n\wedge d_n^{\ast}$.
Then $a_{n+1}\neq 0$ (since $d_n<a_n$) and $d_n\wedge a_{n+1}=0$.
Note that for any $n \in \omega$, $a_n$ cannot be an atom (otherwise it would itself be an atom below $a$).
Moreover, no atom lies below $a_{n+1}$ (any such atom would lie below $a$), so the construction continues.
Thus $\langle d_n:n\in\omega\rangle$ is a family of pairwise disjoint nonzero elements.

Let
\[
e:=\Bigl(\bigvee_{n\in\omega} d_n\Bigr)^{\ast}.
\]
Define the index set
\[
I:=\begin{cases}
\omega, & \text{if } e=0,\\
\omega\cup\{\infty\}, & \text{if } e\neq 0,
\end{cases}
\]
and define elements $c_i\in B$ by $c_n:=d_n$ for $n\in\omega$, and, when $e\neq 0$, let $c_{\infty}:=e$.
Then the family $\{c_i:i\in I\}$ is pairwise disjoint, every $c_i$ is nonzero, and
\[
\bigvee_{i\in I} c_i=1.
\]
Now define
\[
\iota:\Pow(I)\to B,
\qquad
\iota(S):=\bigvee_{i\in S} c_i.
\]
We claim that $\iota$ is a Boolean algebra isomorphism onto its range.

To see that $\iota$ is injective, suppose $\iota(S)=\iota(T)$ and let $k\in S\setminus T$.
Then
\[
c_k\leq \bigvee_{i\in S} c_i = \bigvee_{i\in T} c_i,
\]
so, by pairwise disjointness,
\[
c_k=c_k\wedge\bigvee_{i\in T} c_i=\bigvee_{i\in T}(c_k\wedge c_i)=0,
\]
a contradiction. Hence $S\setminus T=\varnothing$, and by symmetry $T\setminus S=\varnothing$, so $S=T$.

Because the family $\{c_i:i\in I\}$ is pairwise disjoint and joins to $1$, we have for every $S\subseteq I$,
\[
\iota(I\setminus S)=\bigvee_{i\in I\setminus S} c_i
=\Bigl(\bigvee_{i\in S} c_i\Bigr)^{\ast}
=(\iota(S))^{\ast}.
\]
Also, for every family $\{S_j:j\in J\}\subseteq \Pow(I)$,
\[
\iota\Bigl(\bigcup_{j\in J} S_j\Bigr)=\bigvee_{j\in J}\iota(S_j),
\]
because both sides are the join of the same subfamily of the pairwise disjoint family $\{c_i:i\in I\}$.
Moreover,
\begin{align*}
\bigwedge_{j\in J}\iota(S_j)
&=\Bigl(\bigvee_{j\in J}(\iota(S_j))^{\ast}\Bigr)^{\ast}
 =\Bigl(\bigvee_{j\in J}\iota(I\setminus S_j)\Bigr)^{\ast} \\
&=\Bigl(\iota\Bigl(\bigcup_{j\in J}(I\setminus S_j)\Bigr)\Bigr)^{\ast}
 =\Bigl(\iota\Bigl(I\setminus \bigcap_{j\in J}S_j\Bigr)\Bigr)^{\ast}
 =\iota\Bigl(\bigcap_{j\in J}S_j\Bigr).
\end{align*}
Thus the range
\[
C:=\Bigl\{\,\bigvee_{i\in S} c_i : S\subseteq I\,\Bigr\}
\]
is a complete Boolean subalgebra of $B$, and $\iota$ is an isomorphism from $\Pow(I)$ onto $C$.
Since $I$ is countably infinite, $\Pow(I)\cong \Pow(\omega)$, and therefore $C\cong \Pow(\omega)$.
\end{proof}

\begin{proposition}\label{prop:complete-subalgebra-monotone}
Let $C\subseteq B$ be a complete Boolean subalgebra (with the induced Boolean operations).
Then
\[
\Logsub(B)\ \subseteq\ \Logsub(C).
\]
\end{proposition}

\begin{proof}
Assume $\varphi\notin \Logsub(C)$.
Then there exist a nonempty $W\subseteq C^+$ and a $C$-translation $T_C:\Prop\to\Sent(L_C)$ such that
$M_{C,W,T_C}\not\models \varphi$.

Let $\mathrm{Var}(\varphi)$ be the finite set of propositional variables occurring in $\varphi$. For each
$p\in \mathrm{Var}(\varphi)$, let
\[
a_p:=\llbracket T_C(p)\rrbracket_C\in C,
\]
where the Boolean value is computed in the complete Boolean algebra $C$. By Lemma~\ref{lem:truth-value-surjective},
for each $p\in \mathrm{Var}(\varphi)$ there exists an $L_B$-sentence $\sigma_p$ such that
\[
\llbracket \sigma_p\rrbracket_B=a_p.
\]
Define a $B$-translation $T_B:\Prop\to\Sent(L_B)$ by setting
\[
T_B(p):=\sigma_p \qquad (p\in \mathrm{Var}(\varphi)),
\]
and choosing $T_B(p)$ arbitrarily otherwise. Then for every $p\in \mathrm{Var}(\varphi)$ we have
\[
v_{T_B}(p)=\llbracket T_B(p)\rrbracket_B=\llbracket \sigma_p\rrbracket_B=a_p.
\]
Thus $T_B$ realizes in $B$ exactly the same valuation on the variables occurring in $\varphi$ as $T_C$ realizes in $C$.

Now view $W$ as a nonempty subset of $B^+$. Because $C$ is a Boolean subalgebra of $B$, the order $\le$
and meet $\wedge$ on $W$ computed in $C$ agree with those computed in $B$. Hence the inherited co-consistency relation on
$W$ is the same whether it is computed in $C$ or in $B$.

We claim that for every subformula $\psi$ of $\varphi$ and every $w\in W$,
\[
M_{C,W,T_C},w\models \psi \quad\Longleftrightarrow\quad M_{B,W,T_B},w\models \psi.
\]
The proof is by induction on the complexity of $\psi$. If $\psi$ is atomic, say $\psi=p\in \mathrm{Var}(\varphi)$, then
\[
M_{C,W,T_C},w\models p
\iff w\le a_p
\iff w\le v_{T_B}(p)
\iff M_{B,W,T_B},w\models p,
\]
since the order on $W$ is the same in $C$ and in $B$. The Boolean cases are immediate from the induction hypothesis. For
$\Box\chi$ and $\Diamond\chi$, the induction step follows because the accessibility relation $R$ on $W$ is the same in both
models.

Applying the claim to $\psi=\varphi$, we obtain $M_{B,W,T_B}\not\models \varphi$. Hence
$\varphi\notin \Logsub(B)$, as required.
\end{proof}

\begin{theorem}\label{thm:infinite-algebras-KTB}
If $B$ is an infinite complete Boolean algebra, then
\[
\Logsub(B)=\KTB.
\]
\end{theorem}

\begin{proof}
By Theorem~\ref{thm:KTB-completeness} (soundness), $\KTB\subseteq \Logsub(B)$ for every $B$.
For the reverse inclusion, let $C\subseteq B$ be a complete subalgebra with $C\cong \Pow(\omega)$ as in
Lemma~\ref{lem:contains-Pow-omega}.
By Proposition~\ref{prop:complete-subalgebra-monotone},
\[
\Logsub(B)\ \subseteq\ \Logsub(C)
\ =\ \Logsub(\Pow(\omega))
\ =\ \KTB,
\]
where the last equality is Theorem~\ref{thm:PowX-KTB}.
\end{proof}

\subsection{Finite complete algebras}\label{subsec:finite-algebras}

We now turn our attention to the finite case, since the infinite one is already solved.

\begin{lemma}\label{lem:finite-powerset-algebra}
If $B$ is a finite complete Boolean algebra, then $B$ is atomic and
\[
B\ \cong\ \Pow(n)
\]
where $n$ is the number of atoms of $B$. \qed
\end{lemma}

\begin{corollary}\label{cor:finite-reduction-decidable}
Let $B$ be a finite complete Boolean algebra with $n$ atoms. Then
\[
\Logsub(B)=\Logsub(\Pow(n)).
\]
In particular, $\Logsub(B)$ is decidable.
\end{corollary}

\begin{proof}
The first claim follows from Lemma~\ref{lem:finite-powerset-algebra} together with invariance of validity under
Boolean isomorphisms (transporting states and Boolean values along the isomorphism).
For decidability, note that $\Pow(n)$ is finite, so for any fixed formula $\varphi$ there are only finitely many
choices of nonempty $W\subseteq \Pow(n)^+$ and finitely many choices of Boolean values $v_T(p)\in\Pow(n)$ for the finitely
many variables occurring in $\varphi$. Hence one can decide membership of $\varphi$ in $\Logsub(\Pow(n))$
by brute force model checking.
\end{proof}

Interestingly, the finite case generates a whole hierarchy of logics which approximate $\KTB$. 

\begin{lemma}\label{lem:pow-monotonicity}
If $n\le m$, then $\Pow(n)$ is isomorphic to a complete Boolean subalgebra of $\Pow(m)$.
Consequently,
\[
\Logsub(\Pow(m))\ \subseteq\ \Logsub(\Pow(n)).
\]
\end{lemma}

\begin{proof}
Choose a partition of an $m$-element set into $n$ nonempty blocks
\[
X_0,\dots,X_{n-1}.
\]
(For example, if we identify $m$ with $\{0,\dots,m-1\}$, we may take
$X_i=\{i\}$ for $i<n-1$ and $X_{n-1}=\{n-1,\dots,m-1\}$.)
Define
\[
e:\Pow(n)\to \Pow(m),\qquad e(S):=\bigcup_{i\in S} X_i.
\]
Then $e$ preserves arbitrary unions, intersections, complements, $0$, and $1$, so it is a complete
Boolean embedding. Its range
\[
C:=\left\{\bigcup_{i\in S}X_i : S\subseteq n\right\}
\]
is therefore a complete Boolean subalgebra of $\Pow(m)$, and $e$ is a Boolean isomorphism
$\Pow(n)\cong C$.

By Proposition~\ref{prop:complete-subalgebra-monotone},
\[
\Logsub(\Pow(m))\ \subseteq\ \Logsub(C).
\]
Since $C\cong \Pow(n)$, validity is invariant under Boolean isomorphism, and hence
\[
\Logsub(C)=\Logsub(\Pow(n)).
\]
Therefore
\[
\Logsub(\Pow(m))\ \subseteq\ \Logsub(\Pow(n)).
\]
\end{proof}

Together with Theorem~\ref{thm:infinite-algebras-KTB}, this yields a descending chain
\[
\Logsub(\Pow(1))\ \supseteq\ \Logsub(\Pow(2))\ \supseteq\ \cdots\ \supseteq\ \KTB,
\]
whose limit is $\KTB$.

\medskip

Since $\Pow(n)^+$ is finite, Corollary~\ref{cor:finite-reduction-decidable} already shows that
$\Logsub(\Pow(n))$ is decidable. We do not attempt here to give a complete
axiomatization for each fixed $n$. Nevertheless, Jankov--Fine characteristic formulas provide a
useful way to exhibit \emph{additional} principles belonging to $\Logsub(\Pow(n))$.

More precisely, let $F=(W_F,R_F,r)$ be a finite rooted frame, and let $\chi_F$ be its
Jankov--Fine formula. If $F$ cannot occur as a generated subframe of any induced subframe of
$(\Pow(n)^+,R)$, then $\chi_F$ is valid in every induced subframe of $(\Pow(n)^+,R)$, and hence
\[
\chi_F\in \Logsub(\Pow(n)).
\]
Thus characteristic formulas give a systematic method for proving that
$\Logsub(\Pow(n))$ properly extends $\mathsf{KTB}$. 



More concretly, let $F=(W_F,R_F,r)$ be a finite \emph{rooted} frame.
Introduce propositional variables $\{p_w:w\in W_F\}$ and consider the ``diagram'' formula
\begin{align*}
\delta_F :=\ &p_r\ \wedge\ \bigwedge_{u\neq v}\Box\neg(p_u\wedge p_v)\ \wedge\ \bigwedge_{u\in W_F}\Box\Bigl(p_u\to\Box\bigvee_{v\in W_F}p_v\Bigr)\\
&\wedge\ \bigwedge_{u\in W_F}\Box\Bigl(p_u\to\bigl(\ \bigwedge_{uR_F v}\Diamond p_v\ \wedge\ \bigwedge_{\neg(uR_F v)}\neg\Diamond p_v\ \bigr)\Bigr).
\end{align*}
Intuitively, $\delta_F$ asserts that the $p_w$-worlds form a generated copy of $F$ with root satisfying $p_r$.
The corresponding Jankov--Fine formula is
\[\chi_F:=\neg\delta_F.\]
Then $\chi_F$ is valid on a frame $G$ iff $F$ is \emph{not} present as a generated subframe of $G$ (equivalently,
$F$ is not a bounded-morphic image of a generated subframe of $G$).

\smallskip
\noindent\emph{Example  ($n=2$): excluding a triangle.}
In $(\Pow(2)^+,R)$ there is no triple of subsets of $\{1,2\}$ that are pairwise intersecting.
Consequently, the reflexive triangle frame $K_3^\mathrm{r}$ (three worlds, each related to itself and to the other two)
\emph{cannot} occur as a subframe of $(\Pow(2)^+,R)$.
Hence $\chi_{K_3^\mathrm{r}}\in\Logsub(\Pow(2))$.
In this case, since $K_3^\mathrm{r}$ has full accessibility relation, the diagram formula simplifies to
\[
\delta_{K_3^\mathrm{r}} = p_0\ \wedge\ \bigwedge_{0\le i<j\le 2}\Box\neg(p_i\wedge p_j)\ \wedge\ \bigwedge_{i=0}^2\Box\Bigl(p_i\to (\Diamond p_0\wedge\Diamond p_1\wedge\Diamond p_2)\Bigr)\ \wedge\ \bigwedge_{i=0}^2\Box\Bigl(p_i\to\Box(p_0\vee p_1\vee p_2)\Bigr),
\]
so $\chi_{K_3^\mathrm{r}}=\neg\delta_{K_3^\mathrm{r}}$ is an explicit additional axiom over $\KTB$.

\section{Generic ultrafilters and the external forcing modality}\label{sec:generic-forcing}

The internal semantics developed in Sections~\ref{sec:extended-semantics}--\ref{sec:results} uses the
\emph{compatibility} (or \emph{co-consistency}) relation on a fixed complete Boolean algebra $B$ and
thereby yields an all-state ``internal'' modality. In the usual forcing construction, however,
one passes from a Boolean-valued universe to a classical two-valued model by quotienting
with a \emph{generic} ultrafilter, and one then studies modality \emph{externally} by moving between
models via forcing extensions (Hamkins--L\"owe \cite{HamkinsLowe}).

In this section we make the comparison precise at the level of \emph{set-theoretic sentences} and
explain why the external forcing modality gives rise to the modal system $\mathsf{S4.2}$, in sharp
contrast with the internal compatibility modality. In the all-state semantics of Sections~\ref{sec:extended-semantics}--\ref{sec:results},
the isolated state $0$ forces failures of reflexive principles such as $(T)$. By contrast, in the
translation-based nonzero-state semantics of Sections~\ref{sec:completeness}--\ref{sec:specific-boolean-logics},
the exact global logic is $\KTB$.

\subsection{Generic ultrafilters and forcing extensions}\label{subsec:generic-ultrafilters}

Fix a transitive ground model $M\models\mathsf{ZFC}$ and a complete Boolean algebra $B\in M$.
A filter $G\subseteq B$ is \emph{$M$-generic} if it meets every dense subset of $B$ belonging to $M$.
(In the Boolean-algebra presentation of forcing, such a generic filter is automatically an ultrafilter.)
Whenever $G$ is $M$-generic, the quotient $M^{(B)}/G$ is (canonically) isomorphic to the usual forcing
extension $M[G]$.

We will use the Boolean-valued quotient semantics of Definition~\ref{def:ultrafilter-quotient}.
Theorem~\ref{thm:ultrafilter-truth-lemma} then specializes to the familiar forcing ``Truth Lemma'':
for every $L_{\in}$-sentence $\sigma$,
\[
M^{(B)}/G \models \sigma
\quad\Longleftrightarrow\quad
\val{\sigma}\in G.
\]

\begin{proposition}\label{prop:boolean-forcing}
Let $M\models\mathsf{ZFC}$ be a transitive ground model, let $B\in M$ be a complete Boolean algebra,
and let $\sigma\in\Sent(\LP)$.
For every $b\in B$ the following hold:
\begin{enumerate}[label=(\alph*)]
\item $b\le \val{\sigma}$ iff for every $M$-generic ultrafilter $G\subseteq B$ with $b\in G$ we have
$M^{(B)}/G\models \sigma$.
\item $b\wedge \val{\sigma}\neq 0$ iff there exists an $M$-generic ultrafilter $G\subseteq B$ with
$b\in G$ and $M^{(B)}/G\models \sigma$.
\end{enumerate}
\end{proposition}

\begin{proof}
(a) Suppose $b\le \val{\sigma}$ and let $G$ be $M$-generic with $b\in G$. Since $G$ is upward closed,
$\val{\sigma}\in G$, and hence $M^{(B)}/G\models \sigma$ by Theorem~\ref{thm:ultrafilter-truth-lemma}.
Conversely, if $b\not\le \val{\sigma}$ then $c:=b\wedge \val{\neg\sigma}\neq 0$. By standard forcing
arguments (forcing below $c$) there is an $M$-generic $G$ with $c\in G$, hence $b\in G$ and
$\val{\neg\sigma}\in G$, so $M^{(B)}/G\models\neg\sigma$ by the Truth Lemma, contradicting the
assumption.

(b) If $b\wedge \val{\sigma}\neq 0$, let $c:=b\wedge \val{\sigma}$. Forcing below $c$ produces an
$M$-generic ultrafilter $G$ with $c\in G$, hence $b\in G$ and $\val{\sigma}\in G$, so
$M^{(B)}/G\models\sigma$. Conversely, if such $G$ exists then $b,\val{\sigma}\in G$, hence
$b\wedge \val{\sigma}\in G$ and therefore $b\wedge \val{\sigma}\neq 0$.
\end{proof}

\begin{corollary}\label{cor:atomic-forceability}
Let $p$ be a propositional variable corresponding to an $L_{\in}$-sentence.
For every $b\in B$,
\[
M_B,b\models \Diamond p
\quad\Longleftrightarrow\quad
\exists\text{ $M$-generic ultrafilter $G\subseteq B$ with $b\in G$ and }M^{(B)}/G\models p.
\]
\end{corollary}

\begin{proof}
By Definition~\ref{def:MB}, $M_B,b\models \Diamond p$ iff there exists $a$ with $bRa$ and
$a\le\val{p}$. Using Remark~\ref{rem:R-equivalences}, $bRa$ implies $b\wedge a\neq 0$, hence
$b\wedge \val{p}\neq 0$. Now apply Proposition~\ref{prop:boolean-forcing}(b).
Conversely, if some $M$-generic $G$ contains $b$ and makes $p$ true, then $\val{p}\in G$ by the Truth
Lemma, so $b\wedge \val{p}\neq 0$, and taking $a:=b\wedge \val{p}$ yields a witness for
$M_B,b\models\Diamond p$.
\end{proof}

\subsection{Forcing potentialism and S4.2}\label{subsec:forcing-S4.2}

Hamkins and L\"owe \cite{HamkinsLowe} analyze the \emph{external} forcing modality in which the worlds
are (set-theoretic) models and accessibility is given by \emph{forcing extension}.
Concretely, fix a transitive ground model $M\models\mathsf{ZFC}$ and let $\mathcal W$ be the collection
of all forcing extensions of $M$:
\[
\mathcal W\ :=\ \{\,M[G]: \text{$G$ is $M$-generic for some forcing notion in $M$}\,\}.
\]
Define $N\preccurlyeq N'$ for $N,N'\in\mathcal W$ iff $N'$ is a forcing extension of $N$.
The forcing extension relation is:
\begin{itemize}[leftmargin=2em]
\item \emph{reflexive} (trivial forcing),
\item \emph{transitive} (iterated forcing), and
\item \emph{directed} (any two extensions have a common further extension, e.g.\ by product forcing).
\end{itemize}
The directedness property is exactly the relational condition corresponding to the modal axiom $(.2)$.

\begin{definition}\label{def:external-forcing-modality}
Let $\varphi$ be a modal formula whose propositional variables are interpreted as $L_{\in}$-sentences.
In the Kripke model $(\mathcal W,\preccurlyeq,V)$ where $V(p)=\{N\in\mathcal W: N\models p\}$, we write
\[
M\models \Box_F \varphi
\]
to mean that $\varphi$ holds in \emph{all} forcing extensions of $M$, and
\[
M\models \Diamond_F \varphi
\]
to mean that $\varphi$ holds in \emph{some} forcing extension of $M$.
\end{definition}

\begin{proposition}[\cite{HamkinsLowe} ]\label{prop:S4.2-soundness}
The external forcing modality validates the axioms $(K)$, $(T)$, $(4)$ and $(.2)$, and hence all
theorems of $\mathsf{S4.2}$.
\end{proposition}

\begin{proof}
Axiom $(K)$ is valid in every Kripke model.
Axiom $(T)$ holds because $\preccurlyeq$ is reflexive.
Axiom $(4)$ holds because $\preccurlyeq$ is transitive.
For $(.2)$, assume $M\models \Diamond_F\Box_F\psi$, so there is an extension $N\succeq M$ with
$N\models \Box_F\psi$. Let $N'\succeq M$ be arbitrary. By directedness there is a common extension
$K\succeq N,N'$. Since $N\models \Box_F\psi$, we have $K\models\psi$, whence $N'\models\Diamond_F\psi$.
Thus $M\models \Box_F\Diamond_F\psi$.
\end{proof}

\begin{theorem}[\cite{HamkinsLowe} ]\label{thm:HamkinsLowe-S4.2}
Assuming $\mathsf{ZFC}$ is consistent, the $\mathsf{ZFC}$-provable modal principles valid under the
external forcing interpretation are exactly those of $\mathsf{S4.2}$.
\end{theorem}

Corollary~\ref{cor:atomic-forceability} shows that, for \emph{non-modal} set-theoretic sentences,
our internal compatibility possibility at a Boolean state $b$ coincides with the usual forcing notion of
\emph{forceability below $b$}. In particular, at the top element $1_B$ we recover the familiar slogan:
\[
M_B,1_B\models \Diamond p
\quad\Longleftrightarrow\quad
\text{$p$ holds in some forcing extension.}
\]

Nevertheless, the two modalities diverge sharply once one allows \emph{nested} modal operators.
In the all-state semantics of Sections~\ref{sec:extended-semantics}--\ref{sec:results}, the isolated state
$0$ forces failures of reflexive principles such as $(T)$. By contrast, in the translation-based semantics
of Section~\ref{sec:completeness} one works on the nonzero part $B^+$, and the exact global logic becomes
$\KTB$. Even there, however, the compatibility relation is in general non-transitive, so
principles such as $(4)$  still fail in general. Thus the internal nonzero-state logic is $\KTB$,
whereas the external forcing logic is $\mathsf{S4.2}$.


\begin{thebibliography}{99}

\bibitem{BartonWilliams}
N.~Barton and K.~J.~Williams.
\newblock Varieties of class-theoretic potentialism.
\newblock \emph{The Review of Symbolic Logic}, 17(1):272--304, 2024.

\bibitem{Bell}
J.~L.~Bell.
\newblock \emph{Boolean-Valued Models and Independence Proofs in Set Theory}.
\newblock Clarendon Press, Oxford, 1977.

\bibitem{Blackburn}
P.~Blackburn, M.~de Rijke, and Y.~Venema.
\newblock \emph{Modal Logic}.
\newblock Cambridge University Press, 2001.

\bibitem{EsakiaLoewe2012}
L. Esakia and B. Löwe,
\newblock Fatal Heyting Algebras and Forcing Persistent Sentences.
\newblock \emph{Studia Logica} 100(1--2) (2012), 163--173.

\bibitem{GilbertVenturi2021}
D. R.\ Gilbert and G. Venturi,
\newblock Reflexive-Insensitive Logics, the Boxdot Translation, and the Modal Logic of Generic Absoluteness.
\newblock \emph{Notre Dame Journal of Formal Logic} 62(2) (2021), 269--283.

\bibitem{HamkinsLeibmanLow}
J.~D.~Hamkins, G.~Leibman, and B.~L\"owe.
\newblock Structural connections between a forcing class and its modal logic.
\newblock \emph{Israel Journal of Mathematics}, 207(2):617--651, 2015.

\bibitem{HamkinsLinnebo}
J.~D.~Hamkins and \O.~Linnebo.
\newblock The modal logic of set-theoretic potentialism and the potentialist maximality principles.
\newblock \emph{Review of Symbolic Logic}, 15(1):1--35, 2022.

\bibitem{HamkinsLowe}
J.~D.~Hamkins and B.~L\"owe.
\newblock The modal logic of forcing.
\newblock \emph{Transactions of the American Mathematical Society}, 360(4):1793--1817, 2008.

\bibitem{Jech}
T.~Jech.
\newblock \emph{Set Theory}.
\newblock Springer Monographs in Mathematics. Springer, 2002.

\bibitem{InamdarLoewe2016}
T. Inamdar and B. Löwe,
\newblock The Modal Logic of Inner Models.
\newblock \emph{The Journal of Symbolic Logic} 81(1) (2016), 225--236.

\bibitem{LowePassmannTarafder}
B.~L\"owe, R.~Pa{\ss}mann, and S.~Tarafder.
\newblock Constructing illoyal algebra-valued models of set theory.
\newblock \emph{Algebra Universalis}, 82(46), 2021.

\end{thebibliography}
\end{document}